\documentclass[12pt,reqno]{article}

\usepackage[usenames]{color}
\usepackage{amssymb}
\usepackage{amsthm}
\usepackage{graphicx}
\usepackage{amscd}
\usepackage{ifthen}
\usepackage[all]{ xy }
\usepackage{ adjustbox }

\usepackage{ tikz }
\usetikzlibrary{ calc }
\tikzset{
Square/.style = {
inner sep = 0 pt,
minimum width = 8 mm,
minimum height = 8 mm,
draw = black,
fill = none,
align = center
}}

\def\mR{\mathbb{R}}
\def\tn{\textnormal}
\newcommand{\ignore}[1]{}
\newcommand{\set}[1]{\left\{#1\right\}}
\def\D{\mathcal{D}}

\def\T{\mathcal{T}}

\def\bigcdot{\bullet}

\usepackage[colorlinks=true,
linkcolor=webgreen,
filecolor=webbrown,
citecolor=webgreen]{hyperref}

\definecolor{webgreen}{rgb}{0,.5,0}
\definecolor{webbrown}{rgb}{.6,0,0}

\usepackage{color}
\usepackage{fullpage}
\usepackage{float}

\usepackage{graphics,amsmath,amssymb}

\usepackage{amsfonts}
\usepackage{latexsym}
\usepackage{epsf}

\newcommand{\seqnum}[1]{\href{http://oeis.org/#1}{\underline{#1}}}

\usepackage[noblocks]{authblk}

\allowdisplaybreaks

\begin{document}

\theoremstyle{plain}
\newtheorem{theorem}{Theorem}
\newtheorem{corollary}[theorem]{Corollary}
\newtheorem{lemma}[theorem]{Lemma}
\newtheorem{proposition}[theorem]{Proposition}

\theoremstyle{definition}
\newtheorem{definition}[theorem]{Definition}
\newtheorem{example}[theorem]{Example}
\newtheorem{conjecture}[theorem]{Conjecture}
\newtheorem{notation}[theorem]{Notation}
\newtheorem{remark}[theorem]{Remark}

\theoremstyle{remark}

\begin{center}
\vskip 1cm{\LARGE\bf 
Enumeration and Asymptotic Formulas  \\
\vskip .1in
for Rectangular Partitions of the Hypercube
}
\vskip 1cm
\large
Yu Hin (Gary) Au\\
Department of Mathematics and Statistics\\
University of Saskatchewan\\
Saskatoon, SK\\
Canada \\
\href{mailto:au@math.usask.ca}{\tt au@math.usask.ca} \\
\ \\
Fatemeh Bagherzadeh\\
Department of Mathematics and Statistics\\
University of Saskatchewan\\
Saskatoon, SK\\
Canada \\
\href{mailto: bagherzadeh@math.usask.ca}{\tt bagherzadeh@math.usask.ca} \\
\ \\
Murray R. Bremner\\
Department of Mathematics and Statistics\\
University of Saskatchewan\\
Saskatoon, SK\\
Canada \\
\href{mailto:bremner@math.usask.ca}{\tt bremner@math.usask.ca} \\
\ \\
\end{center}

\begin{abstract}
We study a two-parameter generalization of the Catalan numbers
$C_{d,p}(n)$, which counts the number of ways to subdivide the $d$-dimensional hypercube into 
$n$ rectangular regions using orthogonal partitions of fixed arity $p$. 
Bremner \& Dotsenko first introduced the numbers $C_{d,p}(n)$ in their work on
tensor products of operads, wherein they used homological algebra to prove 
a recursive formula and a functional equation.
We express $C_{d,p}(n)$ as simple finite sums, and
determine their growth rate and asymptotic behavior. 
We give an elementary combinatorial proof of the functional equation, as well as a bijection between hypercube decompositions and a family of full $p$-ary trees.
Our results generalize the well-known correspondence 
between Catalan numbers and full binary trees.
\end{abstract}

\section{Introduction}\label{sec1}

\subsection{Catalan numbers}\label{subsec11}

The (binary) Catalan numbers form a well-known and ubiquitous integer sequence
\footnote{We index these numbers starting at $n=1$ rather than the more conventional
$n=0$ for reasons that will become apparent in the later sections.}:
\[
C(n) = \frac{1}{n} \binom{2n{-}2}{n{-}1} \qquad (n \ge 1).
\]
These numbers have over 200 different combinatorial interpretations; see Stanley \cite{Stanley} and sequence \seqnum{A000108} from the On-Line Encyclopedia of Integer Sequences (OEIS)~\cite{OEIS}.
We focus on the following three:
\begin{itemize}
\item[(i)]
Given a set with a binary operation, 
$C(n)$ counts the ways to parenthesize 
a sequence with $n{-}1$ operations and $n$ factors.
For example, the $C(4) = 5$ ways to parenthesize a product with $4$ factors 
using 3 operations are as follows:
\[
(((ab)c)d), \qquad
((a(bc))d), \qquad
((ab)(cd)), \qquad
(a((bc)d)), \qquad
(a(b(cd))).
\]
\item[(ii)]
$C(n)$ counts the plane rooted binary trees with $n{-}1$ internal nodes
(including the root) and $n$ leaves,
assuming that every internal node has two children.
For example, the $C(4) = 5$ rooted full binary trees with $4$ leaves
and 3 internal nodes are as follows:
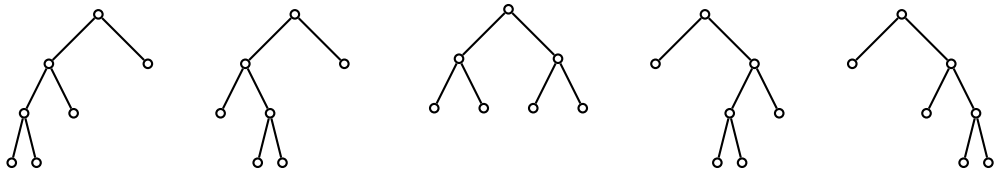
\begin{center}
\begin{tabular}{c@{\qquad}c@{\qquad}c@{\qquad}c@{\qquad}c}
\begin{tikzpicture}
[scale=0.329,thick,main node/.style={circle,inner sep=0.4mm,draw,font=\small\sffamily}]
  \node[main node] at (4,6) (0) {};
  \node[main node] at (2,4) (l){};
  \node[main node] at (6,4) (r){};
  \node[main node] at (1,2) (ll){};
  \node[main node] at (3,2) (lr){};
  \node[main node] at (0.5,0) (lll){};
  \node[main node] at (1.5,0) (llr){};
  \path[every node/.style={font=\sffamily}]
    (0) edge (l)
    (0) edge (r)
    (l) edge (ll)
    (l) edge (lr)
    (ll) edge (lll)
    (ll) edge (llr)
    ;
\end{tikzpicture}
&
\begin{tikzpicture}
[scale=0.329,thick,main node/.style={circle,inner sep=0.4mm,draw,font=\small\sffamily}]
  \node[main node] at (4,6) (0) {};
  \node[main node] at (2,4) (l){};
  \node[main node] at (6,4) (r){};
  \node[main node] at (1,2) (ll){};
  \node[main node] at (3,2) (lr){};
  \node[main node] at (2.5,0) (lrl){};
  \node[main node] at (3.5,0) (lrr){};
  \path[every node/.style={font=\sffamily}]
    (0) edge (l)
    (0) edge (r)
    (l) edge (ll)
    (l) edge (lr)
    (lr) edge (lrl)
    (lr) edge (lrr)
    ;
\end{tikzpicture}
&
\begin{tikzpicture}
[scale=0.329,thick,main node/.style={circle,inner sep=0.4mm,draw,font=\small\sffamily}]
  \node[main node] at (4,6) (0) {};
  \node[main node] at (2,4) (l){};
  \node[main node] at (6,4) (r){};
  \node[main node] at (1,2) (ll){};
  \node[main node] at (3,2) (lr){};
  \node[main node] at (5,2) (rl){};
  \node[main node] at (7,2) (rr){};
  \node at (0.5,0) (lll){};
  \path[every node/.style={font=\sffamily}]
    (0) edge (l)
    (0) edge (r)
    (l) edge (ll)
    (l) edge (lr)
    (r) edge (rl)
    (r) edge (rr)
    ;
\end{tikzpicture}
&
\begin{tikzpicture}
[scale=0.329,thick,main node/.style={circle,inner sep=0.4mm,draw,font=\small\sffamily}]
  \node[main node] at (4,6) (0) {};
  \node[main node] at (2,4) (l){};
  \node[main node] at (6,4) (r){};
  \node[main node] at (5,2) (rl){};
  \node[main node] at (7,2) (rr){};
  \node[main node] at (4.5,0) (rll){};
  \node[main node] at (5.5,0) (rlr){};
  \path[every node/.style={font=\sffamily}]
    (0) edge (l)
    (0) edge (r)
    (r) edge (rl)
    (r) edge (rr)
    (rl) edge (rll)
    (rl) edge (rlr)
    ;
\end{tikzpicture}
&
\begin{tikzpicture}
[scale=0.329,thick,main node/.style={circle,inner sep=0.4mm,draw,font=\small\sffamily}]
  \node[main node] at (4,6) (0) {};
  \node[main node] at (2,4) (l){};
  \node[main node] at (6,4) (r){};
  \node[main node] at (5,2) (rl){};
  \node[main node] at (7,2) (rr){};
  \node[main node] at (6.5,0) (rrl){};
  \node[main node] at (7.5,0) (rrr){};
  \path[every node/.style={font=\sffamily}]
    (0) edge (l)
    (0) edge (r)
    (r) edge (rl)
    (r) edge (rr)
    (rr) edge (rrl)
    (rr) edge (rrr)
    ;
\end{tikzpicture}
\end{tabular}
\end{center}
\item[(iii)]
$C(n)$ counts the dyadic partitions of the unit interval obtained by $n{-}1$ bisections
into $n$ subintervals \cite{Cannon1,Cannon2}.
For example, the $C(4) = 5$ ways to partition the unit interval into 4 subintervals 
using 3 bisections are as follows:
\begin{center}
\begin{tabular}{c@{\quad\quad\;}c@{\quad\quad\;}c@{\quad\quad\;}c@{\quad\quad\;}c}
\begin{tikzpicture}
[scale=0.25,thick]
\def\x{0.65}
\draw (0,0) -- (8,0);
\draw (0,{\x}) -- (0, {-\x});
\draw (8,{\x}) -- (8, {-\x});
\draw (1,{\x}) -- (1, {-\x});
\draw (2,{\x}) -- (2, {-\x});
\draw (4,{\x}) -- (4, {-\x});
\end{tikzpicture}
&
\begin{tikzpicture}
[scale=0.25,thick]
\def\x{0.65}
\draw (0,0) -- (8,0);
\draw (0,{\x}) -- (0, {-\x});
\draw (8,{\x}) -- (8, {-\x});
\draw (3,{\x}) -- (3, {-\x});
\draw (2,{\x}) -- (2, {-\x});
\draw (4,{\x}) -- (4, {-\x});
\end{tikzpicture}
&
\begin{tikzpicture}
[scale=0.25,thick]
\def\x{0.65}
\draw (0,0) -- (8,0);
\draw (0,{\x}) -- (0, {-\x});
\draw (8,{\x}) -- (8, {-\x});
\draw (2,{\x}) -- (2, {-\x});
\draw (6,{\x}) -- (6, {-\x});
\draw (4,{\x}) -- (4, {-\x});
\end{tikzpicture}
&
\begin{tikzpicture}
[scale=0.25,thick]
\def\x{0.65}
\draw (0,0) -- (8,0);
\draw (0,{\x}) -- (0, {-\x});
\draw (8,{\x}) -- (8, {-\x});
\draw (5,{\x}) -- (5, {-\x});
\draw (6,{\x}) -- (6, {-\x});
\draw (4,{\x}) -- (4, {-\x});
\end{tikzpicture}
&
\begin{tikzpicture}
[scale=0.25,thick]
\def\x{0.65}
\draw (0,0) -- (8,0);
\draw (0,{\x}) -- (0, {-\x});
\draw (8,{\x}) -- (8, {-\x});
\draw (7,{\x}) -- (7, {-\x});
\draw (6,{\x}) -- (6, {-\x});
\draw (4,{\x}) -- (4, {-\x});
\end{tikzpicture}
\end{tabular}
\end{center}
\end{itemize}
If we write $y$ for the generating function of $C(n)$,
\[
y = \sum_{n \ge 1} C(n) x^n = 
x + x^2 + 2 x^3 + 5 x^4 + 14 x^5 + 42 x^6 + 132 x^7 + 429 x^8 + \cdots,
\]
then one can check that $y$ satisfies the functional equation
\begin{equation}\label{CatalanFnEq}
x + y^2 = y.
\end{equation}

\subsection{Geometry of higher-dimensional Catalan numbers}
\label{subsec12}

We now describe a higher-dimensional generalization of the Catalan numbers that was first studied in~\cite{BD}. Therein, they analyzed these numbers mainly from the perspective of algebraic operads, but also provided the following combinatorial description in terms of subdividing the $d$-dimensional open unit hypercube 
using a sequence of $p$-ary partitions, generalizing interpretation (iii) 
for the ordinary Catalan numbers.

\begin{definition}\label{defndecomp}
Fix a dimension $d \ge 1$. 
Consider an open $d$-rectangle in the open unit $d$-cube:
\[
R = ( a_1, b_1 ) \times \cdots \times ( a_d, b_d ) \subseteq (0,1)^d.
\]
Fix an arity $p \ge 2$.
For a fixed index $1 \le i \le d$, partition the $i^{\tn{th}}$ interval $(a_i, b_i)$
in the Cartesian product into $p$ equal subintervals with endpoints
\[
c_i(j) = a_i + \frac{j}{p} ( b_i {-} a_i )
\quad\quad\quad
( 0 \le j \le p ).
\]
We define the \emph{$p$-splitting of $R$ in the coordinate $i$} to be
\[
H_i(R) =
\left\{ \,
( a_1, b_1 )
\times \cdots \times
\big(
c_i(j{-}1),
c_i(j)
\big)
\times \cdots \times
( a_d, b_d )
\mid
1 \le j \le p
\, \right\}.
\]
Fix an integer $m \ge 0$ and set $S_0 = \{ (0,1)^d \}$.
For $k = 1, \dots, m$ perform these steps:
\begin{enumerate}
\item
Choose an element $R \in S_{k-1}$.
\item
Choose a direction $1 \le i \le d$.
\item
Set $S_k = \big( S_{k-1} \setminus \{ R \} \big) \cup H_i(R)$.
\end{enumerate}
The result $S_m$ is called a \emph{$(d,p,n)$-decomposition};
by this we mean
a $p$-ary decomposition of the unit $d$-cube into the disjoint union of $n$ blocks
($d$-subrectangles) where $n = 1 + m(p{-}1)$. 
We define $C_{d,p}(n)$ to be the number of distinct $(d,p,n)$-decompositions.
\end{definition}

Figure~\ref{degrees1234} illustrates all $(2,2,n)$-decompositions for $n \leq 4$. 
From the diagrams, we see that $C_{2,2}(n) = 1, 2, 8, 39$ for $n = 1, 2, 3, 4$ respectively.

\begin{figure}[htb]
\begin{center}
\scalebox{0.75}{$
\begin{array}[t]{ll}
\raisebox{0.8cm}{\scalebox{1.25}{$n=1$\quad}}
&
\begin{xy}
(  0,  0 ) = "1";
(  0, 16 ) = "2";
( 16,  0 ) = "3";
( 16, 16 ) = "4";
{ \ar@{-} "1"; "2" };
{ \ar@{-} "3"; "4" };
{ \ar@{-} "1"; "3" };
{ \ar@{-} "2"; "4" };
\end{xy}
\\[1mm]
\raisebox{0.8cm}{\scalebox{1.25}{$n=2$\quad}}
&
\begin{xy}
(  0,  0 ) = "1";
(  0, 16 ) = "2";
(  8,  0 ) = "3";
(  8, 16 ) = "4";
( 16,  0 ) = "5";
( 16, 16 ) = "6";
{ \ar@{-} "1"; "2" };
{ \ar@{-} "3"; "4" };
{ \ar@{-} "5"; "6" };
{ \ar@{-} "1"; "5" };
{ \ar@{-} "2"; "6" };
\end{xy}
\quad
\begin{xy}
(  0,  0 ) = "1";
(  0,  8 ) = "2";
(  0, 16 ) = "3";
( 16,  0 ) = "4";
( 16,  8 ) = "5";
( 16, 16 ) = "6";
{ \ar@{-} "1"; "3" };
{ \ar@{-} "4"; "6" };
{ \ar@{-} "1"; "4" };
{ \ar@{-} "2"; "5" };
{ \ar@{-} "3"; "6" };
\end{xy}
\\[1mm]
\raisebox{0.8cm}{\scalebox{1.25}{$n=3$\quad}}
&
\begin{xy}
(  0,  0 ) = "1";
(  0, 16 ) = "2";
(  4,  0 ) = "3";
(  4, 16 ) = "4";
(  8,  0 ) = "5";
(  8, 16 ) = "6";
( 16,  0 ) = "7";
( 16, 16 ) = "8";
{ \ar@{-} "1"; "2" };
{ \ar@{-} "3"; "4" };
{ \ar@{-} "5"; "6" };
{ \ar@{-} "7"; "8" };
{ \ar@{-} "1"; "7" };
{ \ar@{-} "2"; "8" };
\end{xy}
\quad
\begin{xy}
(  0,  0 ) = "1";
(  0, 16 ) = "2";
(  8,  0 ) = "3";
(  8, 16 ) = "4";
( 12,  0 ) = "5";
( 12, 16 ) = "6";
( 16,  0 ) = "7";
( 16, 16 ) = "8";
{ \ar@{-} "1"; "2" };
{ \ar@{-} "3"; "4" };
{ \ar@{-} "5"; "6" };
{ \ar@{-} "7"; "8" };
{ \ar@{-} "1"; "7" };
{ \ar@{-} "2"; "8" };
\end{xy}
\quad
\begin{xy}
(  0,  0 ) = "1";
(  0,  4 ) = "2";
(  0,  8 ) = "3";
(  0, 16 ) = "4";
( 16,  0 ) = "5";
( 16,  4 ) = "6";
( 16,  8 ) = "7";
( 16, 16 ) = "8";
{ \ar@{-} "1"; "4" };
{ \ar@{-} "5"; "8" };
{ \ar@{-} "1"; "5" };
{ \ar@{-} "2"; "6" };
{ \ar@{-} "3"; "7" };
{ \ar@{-} "4"; "8" };
\end{xy}
\quad
\begin{xy}
(  0,  0 ) = "1";
(  0,  8 ) = "2";
(  0, 12 ) = "3";
(  0, 16 ) = "4";
( 16,  0 ) = "5";
( 16,  8 ) = "6";
( 16, 12 ) = "7";
( 16, 16 ) = "8";
{ \ar@{-} "1"; "4" };
{ \ar@{-} "5"; "8" };
{ \ar@{-} "1"; "5" };
{ \ar@{-} "2"; "6" };
{ \ar@{-} "3"; "7" };
{ \ar@{-} "4"; "8" };
\end{xy}
\quad
\begin{xy}
(  0,  0 ) = "1";
(  0,  8 ) = "2";
(  0, 16 ) = "3";
(  8,  0 ) = "4";
(  8,  8 ) = "5";
(  8, 16 ) = "6";
( 16,  0 ) = "7";
( 16, 16 ) = "8";
{ \ar@{-} "1"; "3" };
{ \ar@{-} "4"; "6" };
{ \ar@{-} "7"; "8" };
{ \ar@{-} "1"; "7" };
{ \ar@{-} "2"; "5" };
{ \ar@{-} "3"; "8" };
\end{xy}
\quad
\begin{xy}
(  0,  0 ) = "1";
(  0, 16 ) = "2";
(  8,  0 ) = "3";
(  8,  8 ) = "4";
(  8, 16 ) = "5";
( 16,  0 ) = "6";
( 16,  8 ) = "7";
( 16, 16 ) = "8";
{ \ar@{-} "1"; "2" };
{ \ar@{-} "3"; "5" };
{ \ar@{-} "6"; "8" };
{ \ar@{-} "1"; "6" };
{ \ar@{-} "2"; "8" };
{ \ar@{-} "4"; "7" };
\end{xy}
\quad
\begin{xy}
(  0,  0 ) = "1";
(  0,  8 ) = "2";
(  0, 16 ) = "3";
(  8,  0 ) = "4";
(  8,  8 ) = "5";
( 16,  0 ) = "6";
( 16,  8 ) = "7";
( 16, 16 ) = "8";
{ \ar@{-} "1"; "3" };
{ \ar@{-} "4"; "5" };
{ \ar@{-} "6"; "8" };
{ \ar@{-} "1"; "6" };
{ \ar@{-} "2"; "7" };
{ \ar@{-} "3"; "8" };
\end{xy}
\quad
\begin{xy}
(  0,  0 ) = "1";
(  0,  8 ) = "2";
(  0, 16 ) = "3";
(  8,  8 ) = "4";
(  8, 16 ) = "5";
( 16,  0 ) = "6";
( 16,  8 ) = "7";
( 16, 16 ) = "8";
{ \ar@{-} "1"; "3" };
{ \ar@{-} "4"; "5" };
{ \ar@{-} "6"; "8" };
{ \ar@{-} "1"; "6" };
{ \ar@{-} "2"; "7" };
{ \ar@{-} "3"; "8" };
\end{xy}
\\[1mm]
\raisebox{0.8cm}{\scalebox{1.25}{$n=4$\quad}}
&
\begin{xy}
(  0,  0 ) = " 1";
(  0, 16 ) = " 2";
(  2,  0 ) = " 3";
(  2, 16 ) = " 4";
{ \ar@{-} " 1"; " 2" };
{ \ar@{-} " 3"; " 4" };
{ \ar@{-} " 1"; " 3" };
{ \ar@{-} " 2"; " 4" };
(  2,  0 ) = " 5";
(  2, 16 ) = " 6";
(  4,  0 ) = " 7";
(  4, 16 ) = " 8";
{ \ar@{-} " 5"; " 6" };
{ \ar@{-} " 7"; " 8" };
{ \ar@{-} " 5"; " 7" };
{ \ar@{-} " 6"; " 8" };
(  4,  0 ) = " 9";
(  4, 16 ) = "10";
(  8,  0 ) = "11";
(  8, 16 ) = "12";
{ \ar@{-} " 9"; "10" };
{ \ar@{-} "11"; "12" };
{ \ar@{-} " 9"; "11" };
{ \ar@{-} "10"; "12" };
(  8,  0 ) = "13";
(  8, 16 ) = "14";
( 16,  0 ) = "15";
( 16, 16 ) = "16";
{ \ar@{-} "13"; "14" };
{ \ar@{-} "15"; "16" };
{ \ar@{-} "13"; "15" };
{ \ar@{-} "14"; "16" };
\end{xy}
\quad
\begin{xy}
(  0,  0 ) = " 1";
(  0,  8 ) = " 2";
(  4,  0 ) = " 3";
(  4,  8 ) = " 4";
{ \ar@{-} " 1"; " 2" };
{ \ar@{-} " 3"; " 4" };
{ \ar@{-} " 1"; " 3" };
{ \ar@{-} " 2"; " 4" };
(  0,  8 ) = " 5";
(  0, 16 ) = " 6";
(  4,  8 ) = " 7";
(  4, 16 ) = " 8";
{ \ar@{-} " 5"; " 6" };
{ \ar@{-} " 7"; " 8" };
{ \ar@{-} " 5"; " 7" };
{ \ar@{-} " 6"; " 8" };
(  4,  0 ) = " 9";
(  4, 16 ) = "10";
(  8,  0 ) = "11";
(  8, 16 ) = "12";
{ \ar@{-} " 9"; "10" };
{ \ar@{-} "11"; "12" };
{ \ar@{-} " 9"; "11" };
{ \ar@{-} "10"; "12" };
(  8,  0 ) = "13";
(  8, 16 ) = "14";
( 16,  0 ) = "15";
( 16, 16 ) = "16";
{ \ar@{-} "13"; "14" };
{ \ar@{-} "15"; "16" };
{ \ar@{-} "13"; "15" };
{ \ar@{-} "14"; "16" };
\end{xy}
\quad
\begin{xy}
(  0,  0 ) = " 1";
(  0,  8 ) = " 2";
(  4,  0 ) = " 3";
(  4,  8 ) = " 4";
{ \ar@{-} " 1"; " 2" };
{ \ar@{-} " 3"; " 4" };
{ \ar@{-} " 1"; " 3" };
{ \ar@{-} " 2"; " 4" };
(  0,  8 ) = " 5";
(  0, 16 ) = " 6";
(  8,  8 ) = " 7";
(  8, 16 ) = " 8";
{ \ar@{-} " 5"; " 6" };
{ \ar@{-} " 7"; " 8" };
{ \ar@{-} " 5"; " 7" };
{ \ar@{-} " 6"; " 8" };
(  4,  0 ) = " 9";
(  4,  8 ) = "10";
(  8,  0 ) = "11";
(  8,  8 ) = "12";
{ \ar@{-} " 9"; "10" };
{ \ar@{-} "11"; "12" };
{ \ar@{-} " 9"; "11" };
{ \ar@{-} "10"; "12" };
(  8,  0 ) = "13";
(  8, 16 ) = "14";
( 16,  0 ) = "15";
( 16, 16 ) = "16";
{ \ar@{-} "13"; "14" };
{ \ar@{-} "15"; "16" };
{ \ar@{-} "13"; "15" };
{ \ar@{-} "14"; "16" };
\end{xy}
\quad
\begin{xy}
(  0,  0 ) = " 1";
(  0,  8 ) = " 2";
(  4,  0 ) = " 3";
(  4,  8 ) = " 4";
{ \ar@{-} " 1"; " 2" };
{ \ar@{-} " 3"; " 4" };
{ \ar@{-} " 1"; " 3" };
{ \ar@{-} " 2"; " 4" };
(  0,  8 ) = " 5";
(  0, 16 ) = " 6";
( 16,  8 ) = " 7";
( 16, 16 ) = " 8";
{ \ar@{-} " 5"; " 6" };
{ \ar@{-} " 7"; " 8" };
{ \ar@{-} " 5"; " 7" };
{ \ar@{-} " 6"; " 8" };
(  4,  0 ) = " 9";
(  4,  8 ) = "10";
(  8,  0 ) = "11";
(  8,  8 ) = "12";
{ \ar@{-} " 9"; "10" };
{ \ar@{-} "11"; "12" };
{ \ar@{-} " 9"; "11" };
{ \ar@{-} "10"; "12" };
(  8,  0 ) = "13";
(  8,  8 ) = "14";
( 16,  0 ) = "15";
( 16,  8 ) = "16";
{ \ar@{-} "13"; "14" };
{ \ar@{-} "15"; "16" };
{ \ar@{-} "13"; "15" };
{ \ar@{-} "14"; "16" };
\end{xy}
\quad
\begin{xy}
(  0,  0 ) = " 1";
(  0, 16 ) = " 2";
(  4,  0 ) = " 3";
(  4, 16 ) = " 4";
{ \ar@{-} " 1"; " 2" };
{ \ar@{-} " 3"; " 4" };
{ \ar@{-} " 1"; " 3" };
{ \ar@{-} " 2"; " 4" };
(  4,  0 ) = " 5";
(  4, 16 ) = " 6";
(  6,  0 ) = " 7";
(  6, 16 ) = " 8";
{ \ar@{-} " 5"; " 6" };
{ \ar@{-} " 7"; " 8" };
{ \ar@{-} " 5"; " 7" };
{ \ar@{-} " 6"; " 8" };
(  6,  0 ) = " 9";
(  6, 16 ) = "10";
(  8,  0 ) = "11";
(  8, 16 ) = "12";
{ \ar@{-} " 9"; "10" };
{ \ar@{-} "11"; "12" };
{ \ar@{-} " 9"; "11" };
{ \ar@{-} "10"; "12" };
(  8,  0 ) = "13";
(  8, 16 ) = "14";
( 16,  0 ) = "15";
( 16, 16 ) = "16";
{ \ar@{-} "13"; "14" };
{ \ar@{-} "15"; "16" };
{ \ar@{-} "13"; "15" };
{ \ar@{-} "14"; "16" };
\end{xy}
\quad
\begin{xy}
(  0,  0 ) = " 1";
(  0, 16 ) = " 2";
(  4,  0 ) = " 3";
(  4, 16 ) = " 4";
{ \ar@{-} " 1"; " 2" };
{ \ar@{-} " 3"; " 4" };
{ \ar@{-} " 1"; " 3" };
{ \ar@{-} " 2"; " 4" };
(  4,  0 ) = " 5";
(  4,  8 ) = " 6";
(  8,  0 ) = " 7";
(  8,  8 ) = " 8";
{ \ar@{-} " 5"; " 6" };
{ \ar@{-} " 7"; " 8" };
{ \ar@{-} " 5"; " 7" };
{ \ar@{-} " 6"; " 8" };
(  4,  8 ) = " 9";
(  4, 16 ) = "10";
(  8,  8 ) = "11";
(  8, 16 ) = "12";
{ \ar@{-} " 9"; "10" };
{ \ar@{-} "11"; "12" };
{ \ar@{-} " 9"; "11" };
{ \ar@{-} "10"; "12" };
(  8,  0 ) = "13";
(  8, 16 ) = "14";
( 16,  0 ) = "15";
( 16, 16 ) = "16";
{ \ar@{-} "13"; "14" };
{ \ar@{-} "15"; "16" };
{ \ar@{-} "13"; "15" };
{ \ar@{-} "14"; "16" };
\end{xy}
\quad
\begin{xy}
(  0,  0 ) = " 1";
(  0, 16 ) = " 2";
(  4,  0 ) = " 3";
(  4, 16 ) = " 4";
{ \ar@{-} " 1"; " 2" };
{ \ar@{-} " 3"; " 4" };
{ \ar@{-} " 1"; " 3" };
{ \ar@{-} " 2"; " 4" };
(  4,  0 ) = " 5";
(  4, 16 ) = " 6";
(  8,  0 ) = " 7";
(  8, 16 ) = " 8";
{ \ar@{-} " 5"; " 6" };
{ \ar@{-} " 7"; " 8" };
{ \ar@{-} " 5"; " 7" };
{ \ar@{-} " 6"; " 8" };
(  8,  0 ) = " 9";
(  8, 16 ) = "10";
( 12,  0 ) = "11";
( 12, 16 ) = "12";
{ \ar@{-} " 9"; "10" };
{ \ar@{-} "11"; "12" };
{ \ar@{-} " 9"; "11" };
{ \ar@{-} "10"; "12" };
( 12,  0 ) = "13";
( 12, 16 ) = "14";
( 16,  0 ) = "15";
( 16, 16 ) = "16";
{ \ar@{-} "13"; "14" };
{ \ar@{-} "15"; "16" };
{ \ar@{-} "13"; "15" };
{ \ar@{-} "14"; "16" };
\end{xy}
\quad
\begin{xy}
(  0,  0 ) = " 1";
(  0, 16 ) = " 2";
(  4,  0 ) = " 3";
(  4, 16 ) = " 4";
{ \ar@{-} " 1"; " 2" };
{ \ar@{-} " 3"; " 4" };
{ \ar@{-} " 1"; " 3" };
{ \ar@{-} " 2"; " 4" };
(  4,  0 ) = " 5";
(  4, 16 ) = " 6";
(  8,  0 ) = " 7";
(  8, 16 ) = " 8";
{ \ar@{-} " 5"; " 6" };
{ \ar@{-} " 7"; " 8" };
{ \ar@{-} " 5"; " 7" };
{ \ar@{-} " 6"; " 8" };
(  8,  0 ) = " 9";
(  8,  8 ) = "10";
( 16,  0 ) = "11";
( 16,  8 ) = "12";
{ \ar@{-} " 9"; "10" };
{ \ar@{-} "11"; "12" };
{ \ar@{-} " 9"; "11" };
{ \ar@{-} "10"; "12" };
(  8,  8 ) = "13";
(  8, 16 ) = "14";
( 16,  8 ) = "15";
( 16, 16 ) = "16";
{ \ar@{-} "13"; "14" };
{ \ar@{-} "15"; "16" };
{ \ar@{-} "13"; "15" };
{ \ar@{-} "14"; "16" };
\end{xy}
\\ [1mm] 
&\begin{xy}
(  0,  8 ) = " 1";
(  0, 16 ) = " 2";
(  4,  8 ) = " 3";
(  4, 16 ) = " 4";
{ \ar@{-} " 1"; " 2" };
{ \ar@{-} " 3"; " 4" };
{ \ar@{-} " 1"; " 3" };
{ \ar@{-} " 2"; " 4" };
(  0,  0 ) = " 5";
(  0,  8 ) = " 6";
(  8,  0 ) = " 7";
(  8,  8 ) = " 8";
{ \ar@{-} " 5"; " 6" };
{ \ar@{-} " 7"; " 8" };
{ \ar@{-} " 5"; " 7" };
{ \ar@{-} " 6"; " 8" };
(  4,  8 ) = " 9";
(  4, 16 ) = "10";
(  8,  8 ) = "11";
(  8, 16 ) = "12";
{ \ar@{-} " 9"; "10" };
{ \ar@{-} "11"; "12" };
{ \ar@{-} " 9"; "11" };
{ \ar@{-} "10"; "12" };
(  8,  0 ) = "13";
(  8, 16 ) = "14";
( 16,  0 ) = "15";
( 16, 16 ) = "16";
{ \ar@{-} "13"; "14" };
{ \ar@{-} "15"; "16" };
{ \ar@{-} "13"; "15" };
{ \ar@{-} "14"; "16" };
\end{xy}
\quad
\begin{xy}
(  0,  8 ) = " 1";
(  0, 16 ) = " 2";
(  4,  8 ) = " 3";
(  4, 16 ) = " 4";
{ \ar@{-} " 1"; " 2" };
{ \ar@{-} " 3"; " 4" };
{ \ar@{-} " 1"; " 3" };
{ \ar@{-} " 2"; " 4" };
(  0,  0 ) = " 5";
(  0,  8 ) = " 6";
( 16,  0 ) = " 7";
( 16,  8 ) = " 8";
{ \ar@{-} " 5"; " 6" };
{ \ar@{-} " 7"; " 8" };
{ \ar@{-} " 5"; " 7" };
{ \ar@{-} " 6"; " 8" };
(  4,  8 ) = " 9";
(  4, 16 ) = "10";
(  8,  8 ) = "11";
(  8, 16 ) = "12";
{ \ar@{-} " 9"; "10" };
{ \ar@{-} "11"; "12" };
{ \ar@{-} " 9"; "11" };
{ \ar@{-} "10"; "12" };
(  8,  8 ) = "13";
(  8, 16 ) = "14";
( 16,  8 ) = "15";
( 16, 16 ) = "16";
{ \ar@{-} "13"; "14" };
{ \ar@{-} "15"; "16" };
{ \ar@{-} "13"; "15" };
{ \ar@{-} "14"; "16" };
\end{xy}
\quad
\begin{xy}
(  0,  0 ) = " 1";
(  0,  4 ) = " 2";
(  8,  0 ) = " 3";
(  8,  4 ) = " 4";
{ \ar@{-} " 1"; " 2" };
{ \ar@{-} " 3"; " 4" };
{ \ar@{-} " 1"; " 3" };
{ \ar@{-} " 2"; " 4" };
(  0,  4 ) = " 5";
(  0,  8 ) = " 6";
(  8,  4 ) = " 7";
(  8,  8 ) = " 8";
{ \ar@{-} " 5"; " 6" };
{ \ar@{-} " 7"; " 8" };
{ \ar@{-} " 5"; " 7" };
{ \ar@{-} " 6"; " 8" };
(  0,  8 ) = " 9";
(  0, 16 ) = "10";
(  8,  8 ) = "11";
(  8, 16 ) = "12";
{ \ar@{-} " 9"; "10" };
{ \ar@{-} "11"; "12" };
{ \ar@{-} " 9"; "11" };
{ \ar@{-} "10"; "12" };
(  8,  0 ) = "13";
(  8, 16 ) = "14";
( 16,  0 ) = "15";
( 16, 16 ) = "16";
{ \ar@{-} "13"; "14" };
{ \ar@{-} "15"; "16" };
{ \ar@{-} "13"; "15" };
{ \ar@{-} "14"; "16" };
\end{xy}
\quad
\begin{xy}
(  0,  0 ) = " 1";
(  0,  4 ) = " 2";
(  8,  0 ) = " 3";
(  8,  4 ) = " 4";
{ \ar@{-} " 1"; " 2" };
{ \ar@{-} " 3"; " 4" };
{ \ar@{-} " 1"; " 3" };
{ \ar@{-} " 2"; " 4" };
(  0,  4 ) = " 5";
(  0,  8 ) = " 6";
(  8,  4 ) = " 7";
(  8,  8 ) = " 8";
{ \ar@{-} " 5"; " 6" };
{ \ar@{-} " 7"; " 8" };
{ \ar@{-} " 5"; " 7" };
{ \ar@{-} " 6"; " 8" };
(  0,  8 ) = " 9";
(  0, 16 ) = "10";
( 16,  8 ) = "11";
( 16, 16 ) = "12";
{ \ar@{-} " 9"; "10" };
{ \ar@{-} "11"; "12" };
{ \ar@{-} " 9"; "11" };
{ \ar@{-} "10"; "12" };
(  8,  0 ) = "13";
(  8,  8 ) = "14";
( 16,  0 ) = "15";
( 16,  8 ) = "16";
{ \ar@{-} "13"; "14" };
{ \ar@{-} "15"; "16" };
{ \ar@{-} "13"; "15" };
{ \ar@{-} "14"; "16" };
\end{xy}
\quad
\begin{xy}
(  0,  0 ) = " 1";
(  0,  4 ) = " 2";
(  8,  0 ) = " 3";
(  8,  4 ) = " 4";
{ \ar@{-} " 1"; " 2" };
{ \ar@{-} " 3"; " 4" };
{ \ar@{-} " 1"; " 3" };
{ \ar@{-} " 2"; " 4" };
(  0,  4 ) = " 5";
(  0,  8 ) = " 6";
( 16,  4 ) = " 7";
( 16,  8 ) = " 8";
{ \ar@{-} " 5"; " 6" };
{ \ar@{-} " 7"; " 8" };
{ \ar@{-} " 5"; " 7" };
{ \ar@{-} " 6"; " 8" };
(  0,  8 ) = " 9";
(  0, 16 ) = "10";
( 16,  8 ) = "11";
( 16, 16 ) = "12";
{ \ar@{-} " 9"; "10" };
{ \ar@{-} "11"; "12" };
{ \ar@{-} " 9"; "11" };
{ \ar@{-} "10"; "12" };
(  8,  0 ) = "13";
(  8,  4 ) = "14";
( 16,  0 ) = "15";
( 16,  4 ) = "16";
{ \ar@{-} "13"; "14" };
{ \ar@{-} "15"; "16" };
{ \ar@{-} "13"; "15" };
{ \ar@{-} "14"; "16" };
\end{xy}
\quad
\begin{xy}
(  0,  0 ) = " 1";
(  0,  8 ) = " 2";
(  8,  0 ) = " 3";
(  8,  8 ) = " 4";
{ \ar@{-} " 1"; " 2" };
{ \ar@{-} " 3"; " 4" };
{ \ar@{-} " 1"; " 3" };
{ \ar@{-} " 2"; " 4" };
(  0,  8 ) = " 5";
(  0, 12 ) = " 6";
(  8,  8 ) = " 7";
(  8, 12 ) = " 8";
{ \ar@{-} " 5"; " 6" };
{ \ar@{-} " 7"; " 8" };
{ \ar@{-} " 5"; " 7" };
{ \ar@{-} " 6"; " 8" };
(  0, 12 ) = " 9";
(  0, 16 ) = "10";
(  8, 12 ) = "11";
(  8, 16 ) = "12";
{ \ar@{-} " 9"; "10" };
{ \ar@{-} "11"; "12" };
{ \ar@{-} " 9"; "11" };
{ \ar@{-} "10"; "12" };
(  8,  0 ) = "13";
(  8, 16 ) = "14";
( 16,  0 ) = "15";
( 16, 16 ) = "16";
{ \ar@{-} "13"; "14" };
{ \ar@{-} "15"; "16" };
{ \ar@{-} "13"; "15" };
{ \ar@{-} "14"; "16" };
\end{xy}
\quad
\begin{xy}
(  0,  0 ) = " 1";
(  0,  8 ) = " 2";
(  8,  0 ) = " 3";
(  8,  8 ) = " 4";
{ \ar@{-} " 1"; " 2" };
{ \ar@{-} " 3"; " 4" };
{ \ar@{-} " 1"; " 3" };
{ \ar@{-} " 2"; " 4" };
(  0,  8 ) = " 5";
(  0, 16 ) = " 6";
(  8,  8 ) = " 7";
(  8, 16 ) = " 8";
{ \ar@{-} " 5"; " 6" };
{ \ar@{-} " 7"; " 8" };
{ \ar@{-} " 5"; " 7" };
{ \ar@{-} " 6"; " 8" };
(  8,  0 ) = " 9";
(  8, 16 ) = "10";
( 12,  0 ) = "11";
( 12, 16 ) = "12";
{ \ar@{-} " 9"; "10" };
{ \ar@{-} "11"; "12" };
{ \ar@{-} " 9"; "11" };
{ \ar@{-} "10"; "12" };
( 12,  0 ) = "13";
( 12, 16 ) = "14";
( 16,  0 ) = "15";
( 16, 16 ) = "16";
{ \ar@{-} "13"; "14" };
{ \ar@{-} "15"; "16" };
{ \ar@{-} "13"; "15" };
{ \ar@{-} "14"; "16" };
\end{xy}
\quad
\begin{xy}
(  0,  0 ) = " 1";
(  0,  8 ) = " 2";
(  8,  0 ) = " 3";
(  8,  8 ) = " 4";
{ \ar@{-} " 1"; " 2" };
{ \ar@{-} " 3"; " 4" };
{ \ar@{-} " 1"; " 3" };
{ \ar@{-} " 2"; " 4" };
(  0,  8 ) = " 5";
(  0, 16 ) = " 6";
(  8,  8 ) = " 7";
(  8, 16 ) = " 8";
{ \ar@{-} " 5"; " 6" };
{ \ar@{-} " 7"; " 8" };
{ \ar@{-} " 5"; " 7" };
{ \ar@{-} " 6"; " 8" };
(  8,  0 ) = " 9";
(  8,  8 ) = "10";
( 16,  0 ) = "11";
( 16,  8 ) = "12";
{ \ar@{-} " 9"; "10" };
{ \ar@{-} "11"; "12" };
{ \ar@{-} " 9"; "11" };
{ \ar@{-} "10"; "12" };
(  8,  8 ) = "13";
(  8, 16 ) = "14";
( 16,  8 ) = "15";
( 16, 16 ) = "16";
{ \ar@{-} "13"; "14" };
{ \ar@{-} "15"; "16" };
{ \ar@{-} "13"; "15" };
{ \ar@{-} "14"; "16" };
\end{xy}
\\ [1mm] 
&\begin{xy}
(  0,  0 ) = " 1";
(  0,  8 ) = " 2";
(  8,  0 ) = " 3";
(  8,  8 ) = " 4";
{ \ar@{-} " 1"; " 2" };
{ \ar@{-} " 3"; " 4" };
{ \ar@{-} " 1"; " 3" };
{ \ar@{-} " 2"; " 4" };
(  0,  8 ) = " 5";
(  0, 12 ) = " 6";
( 16,  8 ) = " 7";
( 16, 12 ) = " 8";
{ \ar@{-} " 5"; " 6" };
{ \ar@{-} " 7"; " 8" };
{ \ar@{-} " 5"; " 7" };
{ \ar@{-} " 6"; " 8" };
(  0, 12 ) = " 9";
(  0, 16 ) = "10";
( 16, 12 ) = "11";
( 16, 16 ) = "12";
{ \ar@{-} " 9"; "10" };
{ \ar@{-} "11"; "12" };
{ \ar@{-} " 9"; "11" };
{ \ar@{-} "10"; "12" };
(  8,  0 ) = "13";
(  8,  8 ) = "14";
( 16,  0 ) = "15";
( 16,  8 ) = "16";
{ \ar@{-} "13"; "14" };
{ \ar@{-} "15"; "16" };
{ \ar@{-} "13"; "15" };
{ \ar@{-} "14"; "16" };
\end{xy}
\quad
\begin{xy}
(  0,  0 ) = " 1";
(  0,  8 ) = " 2";
(  8,  0 ) = " 3";
(  8,  8 ) = " 4";
{ \ar@{-} " 1"; " 2" };
{ \ar@{-} " 3"; " 4" };
{ \ar@{-} " 1"; " 3" };
{ \ar@{-} " 2"; " 4" };
(  0,  8 ) = " 5";
(  0, 16 ) = " 6";
( 16,  8 ) = " 7";
( 16, 16 ) = " 8";
{ \ar@{-} " 5"; " 6" };
{ \ar@{-} " 7"; " 8" };
{ \ar@{-} " 5"; " 7" };
{ \ar@{-} " 6"; " 8" };
(  8,  0 ) = " 9";
(  8,  8 ) = "10";
( 12,  0 ) = "11";
( 12,  8 ) = "12";
{ \ar@{-} " 9"; "10" };
{ \ar@{-} "11"; "12" };
{ \ar@{-} " 9"; "11" };
{ \ar@{-} "10"; "12" };
( 12,  0 ) = "13";
( 12,  8 ) = "14";
( 16,  0 ) = "15";
( 16,  8 ) = "16";
{ \ar@{-} "13"; "14" };
{ \ar@{-} "15"; "16" };
{ \ar@{-} "13"; "15" };
{ \ar@{-} "14"; "16" };
\end{xy}
\quad
\begin{xy}
(  0,  0 ) = " 1";
(  0,  8 ) = " 2";
(  8,  0 ) = " 3";
(  8,  8 ) = " 4";
{ \ar@{-} " 1"; " 2" };
{ \ar@{-} " 3"; " 4" };
{ \ar@{-} " 1"; " 3" };
{ \ar@{-} " 2"; " 4" };
(  0,  8 ) = " 5";
(  0, 16 ) = " 6";
( 16,  8 ) = " 7";
( 16, 16 ) = " 8";
{ \ar@{-} " 5"; " 6" };
{ \ar@{-} " 7"; " 8" };
{ \ar@{-} " 5"; " 7" };
{ \ar@{-} " 6"; " 8" };
(  8,  0 ) = " 9";
(  8,  4 ) = "10";
( 16,  0 ) = "11";
( 16,  4 ) = "12";
{ \ar@{-} " 9"; "10" };
{ \ar@{-} "11"; "12" };
{ \ar@{-} " 9"; "11" };
{ \ar@{-} "10"; "12" };
(  8,  4 ) = "13";
(  8,  8 ) = "14";
( 16,  4 ) = "15";
( 16,  8 ) = "16";
{ \ar@{-} "13"; "14" };
{ \ar@{-} "15"; "16" };
{ \ar@{-} "13"; "15" };
{ \ar@{-} "14"; "16" };
\end{xy}
\quad
\begin{xy}
(  0,  0 ) = " 1";
(  0, 16 ) = " 2";
(  8,  0 ) = " 3";
(  8, 16 ) = " 4";
{ \ar@{-} " 1"; " 2" };
{ \ar@{-} " 3"; " 4" };
{ \ar@{-} " 1"; " 3" };
{ \ar@{-} " 2"; " 4" };
(  8,  0 ) = " 5";
(  8, 16 ) = " 6";
( 10,  0 ) = " 7";
( 10, 16 ) = " 8";
{ \ar@{-} " 5"; " 6" };
{ \ar@{-} " 7"; " 8" };
{ \ar@{-} " 5"; " 7" };
{ \ar@{-} " 6"; " 8" };
( 10,  0 ) = " 9";
( 10, 16 ) = "10";
( 12,  0 ) = "11";
( 12, 16 ) = "12";
{ \ar@{-} " 9"; "10" };
{ \ar@{-} "11"; "12" };
{ \ar@{-} " 9"; "11" };
{ \ar@{-} "10"; "12" };
( 12,  0 ) = "13";
( 12, 16 ) = "14";
( 16,  0 ) = "15";
( 16, 16 ) = "16";
{ \ar@{-} "13"; "14" };
{ \ar@{-} "15"; "16" };
{ \ar@{-} "13"; "15" };
{ \ar@{-} "14"; "16" };
\end{xy}
\quad
\begin{xy}
(  0,  0 ) = " 1";
(  0, 16 ) = " 2";
(  8,  0 ) = " 3";
(  8, 16 ) = " 4";
{ \ar@{-} " 1"; " 2" };
{ \ar@{-} " 3"; " 4" };
{ \ar@{-} " 1"; " 3" };
{ \ar@{-} " 2"; " 4" };
(  8,  0 ) = " 5";
(  8,  8 ) = " 6";
( 12,  0 ) = " 7";
( 12,  8 ) = " 8";
{ \ar@{-} " 5"; " 6" };
{ \ar@{-} " 7"; " 8" };
{ \ar@{-} " 5"; " 7" };
{ \ar@{-} " 6"; " 8" };
(  8,  8 ) = " 9";
(  8, 16 ) = "10";
( 12,  8 ) = "11";
( 12, 16 ) = "12";
{ \ar@{-} " 9"; "10" };
{ \ar@{-} "11"; "12" };
{ \ar@{-} " 9"; "11" };
{ \ar@{-} "10"; "12" };
( 12,  0 ) = "13";
( 12, 16 ) = "14";
( 16,  0 ) = "15";
( 16, 16 ) = "16";
{ \ar@{-} "13"; "14" };
{ \ar@{-} "15"; "16" };
{ \ar@{-} "13"; "15" };
{ \ar@{-} "14"; "16" };
\end{xy}
\quad
\begin{xy}
(  0,  0 ) = " 1";
(  0, 16 ) = " 2";
(  8,  0 ) = " 3";
(  8, 16 ) = " 4";
{ \ar@{-} " 1"; " 2" };
{ \ar@{-} " 3"; " 4" };
{ \ar@{-} " 1"; " 3" };
{ \ar@{-} " 2"; " 4" };
(  8,  0 ) = " 5";
(  8,  8 ) = " 6";
( 12,  0 ) = " 7";
( 12,  8 ) = " 8";
{ \ar@{-} " 5"; " 6" };
{ \ar@{-} " 7"; " 8" };
{ \ar@{-} " 5"; " 7" };
{ \ar@{-} " 6"; " 8" };
(  8,  8 ) = " 9";
(  8, 16 ) = "10";
( 16,  8 ) = "11";
( 16, 16 ) = "12";
{ \ar@{-} " 9"; "10" };
{ \ar@{-} "11"; "12" };
{ \ar@{-} " 9"; "11" };
{ \ar@{-} "10"; "12" };
( 12,  0 ) = "13";
( 12,  8 ) = "14";
( 16,  0 ) = "15";
( 16,  8 ) = "16";
{ \ar@{-} "13"; "14" };
{ \ar@{-} "15"; "16" };
{ \ar@{-} "13"; "15" };
{ \ar@{-} "14"; "16" };
\end{xy}
\quad
\begin{xy}
(  0,  0 ) = " 1";
(  0, 16 ) = " 2";
(  8,  0 ) = " 3";
(  8, 16 ) = " 4";
{ \ar@{-} " 1"; " 2" };
{ \ar@{-} " 3"; " 4" };
{ \ar@{-} " 1"; " 3" };
{ \ar@{-} " 2"; " 4" };
(  8,  0 ) = " 5";
(  8, 16 ) = " 6";
( 12,  0 ) = " 7";
( 12, 16 ) = " 8";
{ \ar@{-} " 5"; " 6" };
{ \ar@{-} " 7"; " 8" };
{ \ar@{-} " 5"; " 7" };
{ \ar@{-} " 6"; " 8" };
( 12,  0 ) = " 9";
( 12, 16 ) = "10";
( 14,  0 ) = "11";
( 14, 16 ) = "12";
{ \ar@{-} " 9"; "10" };
{ \ar@{-} "11"; "12" };
{ \ar@{-} " 9"; "11" };
{ \ar@{-} "10"; "12" };
( 14,  0 ) = "13";
( 14, 16 ) = "14";
( 16,  0 ) = "15";
( 16, 16 ) = "16";
{ \ar@{-} "13"; "14" };
{ \ar@{-} "15"; "16" };
{ \ar@{-} "13"; "15" };
{ \ar@{-} "14"; "16" };
\end{xy}
\quad
\begin{xy}
(  0,  0 ) = " 1";
(  0, 16 ) = " 2";
(  8,  0 ) = " 3";
(  8, 16 ) = " 4";
{ \ar@{-} " 1"; " 2" };
{ \ar@{-} " 3"; " 4" };
{ \ar@{-} " 1"; " 3" };
{ \ar@{-} " 2"; " 4" };
(  8,  0 ) = " 5";
(  8, 16 ) = " 6";
( 12,  0 ) = " 7";
( 12, 16 ) = " 8";
{ \ar@{-} " 5"; " 6" };
{ \ar@{-} " 7"; " 8" };
{ \ar@{-} " 5"; " 7" };
{ \ar@{-} " 6"; " 8" };
( 12,  0 ) = " 9";
( 12,  8 ) = "10";
( 16,  0 ) = "11";
( 16,  8 ) = "12";
{ \ar@{-} " 9"; "10" };
{ \ar@{-} "11"; "12" };
{ \ar@{-} " 9"; "11" };
{ \ar@{-} "10"; "12" };
( 12,  8 ) = "13";
( 12, 16 ) = "14";
( 16,  8 ) = "15";
( 16, 16 ) = "16";
{ \ar@{-} "13"; "14" };
{ \ar@{-} "15"; "16" };
{ \ar@{-} "13"; "15" };
{ \ar@{-} "14"; "16" };
\end{xy}
\\ [1mm] 
&\begin{xy}
(  0,  0 ) = " 1";
(  0, 16 ) = " 2";
(  8,  0 ) = " 3";
(  8, 16 ) = " 4";
{ \ar@{-} " 1"; " 2" };
{ \ar@{-} " 3"; " 4" };
{ \ar@{-} " 1"; " 3" };
{ \ar@{-} " 2"; " 4" };
(  8,  8 ) = " 5";
(  8, 16 ) = " 6";
( 12,  8 ) = " 7";
( 12, 16 ) = " 8";
{ \ar@{-} " 5"; " 6" };
{ \ar@{-} " 7"; " 8" };
{ \ar@{-} " 5"; " 7" };
{ \ar@{-} " 6"; " 8" };
(  8,  0 ) = " 9";
(  8,  8 ) = "10";
( 16,  0 ) = "11";
( 16,  8 ) = "12";
{ \ar@{-} " 9"; "10" };
{ \ar@{-} "11"; "12" };
{ \ar@{-} " 9"; "11" };
{ \ar@{-} "10"; "12" };
( 12,  8 ) = "13";
( 12, 16 ) = "14";
( 16,  8 ) = "15";
( 16, 16 ) = "16";
{ \ar@{-} "13"; "14" };
{ \ar@{-} "15"; "16" };
{ \ar@{-} "13"; "15" };
{ \ar@{-} "14"; "16" };
\end{xy}
\quad
\begin{xy}
(  0,  0 ) = " 1";
(  0, 16 ) = " 2";
(  8,  0 ) = " 3";
(  8, 16 ) = " 4";
{ \ar@{-} " 1"; " 2" };
{ \ar@{-} " 3"; " 4" };
{ \ar@{-} " 1"; " 3" };
{ \ar@{-} " 2"; " 4" };
(  8,  0 ) = " 5";
(  8,  4 ) = " 6";
( 16,  0 ) = " 7";
( 16,  4 ) = " 8";
{ \ar@{-} " 5"; " 6" };
{ \ar@{-} " 7"; " 8" };
{ \ar@{-} " 5"; " 7" };
{ \ar@{-} " 6"; " 8" };
(  8,  4 ) = " 9";
(  8,  8 ) = "10";
( 16,  4 ) = "11";
( 16,  8 ) = "12";
{ \ar@{-} " 9"; "10" };
{ \ar@{-} "11"; "12" };
{ \ar@{-} " 9"; "11" };
{ \ar@{-} "10"; "12" };
(  8,  8 ) = "13";
(  8, 16 ) = "14";
( 16,  8 ) = "15";
( 16, 16 ) = "16";
{ \ar@{-} "13"; "14" };
{ \ar@{-} "15"; "16" };
{ \ar@{-} "13"; "15" };
{ \ar@{-} "14"; "16" };
\end{xy}
\quad
\begin{xy}
(  0,  0 ) = " 1";
(  0, 16 ) = " 2";
(  8,  0 ) = " 3";
(  8, 16 ) = " 4";
{ \ar@{-} " 1"; " 2" };
{ \ar@{-} " 3"; " 4" };
{ \ar@{-} " 1"; " 3" };
{ \ar@{-} " 2"; " 4" };
(  8,  0 ) = " 5";
(  8,  8 ) = " 6";
( 16,  0 ) = " 7";
( 16,  8 ) = " 8";
{ \ar@{-} " 5"; " 6" };
{ \ar@{-} " 7"; " 8" };
{ \ar@{-} " 5"; " 7" };
{ \ar@{-} " 6"; " 8" };
(  8,  8 ) = " 9";
(  8, 12 ) = "10";
( 16,  8 ) = "11";
( 16, 12 ) = "12";
{ \ar@{-} " 9"; "10" };
{ \ar@{-} "11"; "12" };
{ \ar@{-} " 9"; "11" };
{ \ar@{-} "10"; "12" };
(  8, 12 ) = "13";
(  8, 16 ) = "14";
( 16, 12 ) = "15";
( 16, 16 ) = "16";
{ \ar@{-} "13"; "14" };
{ \ar@{-} "15"; "16" };
{ \ar@{-} "13"; "15" };
{ \ar@{-} "14"; "16" };
\end{xy}
\quad
\begin{xy}
(  0,  4 ) = " 1";
(  0,  8 ) = " 2";
(  8,  4 ) = " 3";
(  8,  8 ) = " 4";
{ \ar@{-} " 1"; " 2" };
{ \ar@{-} " 3"; " 4" };
{ \ar@{-} " 1"; " 3" };
{ \ar@{-} " 2"; " 4" };
(  0,  0 ) = " 5";
(  0,  4 ) = " 6";
( 16,  0 ) = " 7";
( 16,  4 ) = " 8";
{ \ar@{-} " 5"; " 6" };
{ \ar@{-} " 7"; " 8" };
{ \ar@{-} " 5"; " 7" };
{ \ar@{-} " 6"; " 8" };
(  0,  8 ) = " 9";
(  0, 16 ) = "10";
( 16,  8 ) = "11";
( 16, 16 ) = "12";
{ \ar@{-} " 9"; "10" };
{ \ar@{-} "11"; "12" };
{ \ar@{-} " 9"; "11" };
{ \ar@{-} "10"; "12" };
(  8,  4 ) = "13";
(  8,  8 ) = "14";
( 16,  4 ) = "15";
( 16,  8 ) = "16";
{ \ar@{-} "13"; "14" };
{ \ar@{-} "15"; "16" };
{ \ar@{-} "13"; "15" };
{ \ar@{-} "14"; "16" };
\end{xy}
\quad
\begin{xy}
(  0,  8 ) = " 1";
(  0, 12 ) = " 2";
(  8,  8 ) = " 3";
(  8, 12 ) = " 4";
{ \ar@{-} " 1"; " 2" };
{ \ar@{-} " 3"; " 4" };
{ \ar@{-} " 1"; " 3" };
{ \ar@{-} " 2"; " 4" };
(  0, 12 ) = " 5";
(  0, 16 ) = " 6";
(  8, 12 ) = " 7";
(  8, 16 ) = " 8";
{ \ar@{-} " 5"; " 6" };
{ \ar@{-} " 7"; " 8" };
{ \ar@{-} " 5"; " 7" };
{ \ar@{-} " 6"; " 8" };
(  0,  0 ) = " 9";
(  0,  8 ) = "10";
( 16,  0 ) = "11";
( 16,  8 ) = "12";
{ \ar@{-} " 9"; "10" };
{ \ar@{-} "11"; "12" };
{ \ar@{-} " 9"; "11" };
{ \ar@{-} "10"; "12" };
(  8,  8 ) = "13";
(  8, 16 ) = "14";
( 16,  8 ) = "15";
( 16, 16 ) = "16";
{ \ar@{-} "13"; "14" };
{ \ar@{-} "15"; "16" };
{ \ar@{-} "13"; "15" };
{ \ar@{-} "14"; "16" };
\end{xy}
\quad
\begin{xy}
(  0,  8 ) = " 1";
(  0, 12 ) = " 2";
(  8,  8 ) = " 3";
(  8, 12 ) = " 4";
{ \ar@{-} " 1"; " 2" };
{ \ar@{-} " 3"; " 4" };
{ \ar@{-} " 1"; " 3" };
{ \ar@{-} " 2"; " 4" };
(  0,  0 ) = " 5";
(  0,  8 ) = " 6";
( 16,  0 ) = " 7";
( 16,  8 ) = " 8";
{ \ar@{-} " 5"; " 6" };
{ \ar@{-} " 7"; " 8" };
{ \ar@{-} " 5"; " 7" };
{ \ar@{-} " 6"; " 8" };
(  0, 12 ) = " 9";
(  0, 16 ) = "10";
( 16, 12 ) = "11";
( 16, 16 ) = "12";
{ \ar@{-} " 9"; "10" };
{ \ar@{-} "11"; "12" };
{ \ar@{-} " 9"; "11" };
{ \ar@{-} "10"; "12" };
(  8,  8 ) = "13";
(  8, 12 ) = "14";
( 16,  8 ) = "15";
( 16, 12 ) = "16";
{ \ar@{-} "13"; "14" };
{ \ar@{-} "15"; "16" };
{ \ar@{-} "13"; "15" };
{ \ar@{-} "14"; "16" };
\end{xy}
\quad
\begin{xy}
(  0,  8 ) = " 1";
(  0, 16 ) = " 2";
(  8,  8 ) = " 3";
(  8, 16 ) = " 4";
{ \ar@{-} " 1"; " 2" };
{ \ar@{-} " 3"; " 4" };
{ \ar@{-} " 1"; " 3" };
{ \ar@{-} " 2"; " 4" };
(  0,  0 ) = " 5";
(  0,  4 ) = " 6";
( 16,  0 ) = " 7";
( 16,  4 ) = " 8";
{ \ar@{-} " 5"; " 6" };
{ \ar@{-} " 7"; " 8" };
{ \ar@{-} " 5"; " 7" };
{ \ar@{-} " 6"; " 8" };
(  0,  4 ) = " 9";
(  0,  8 ) = "10";
( 16,  4 ) = "11";
( 16,  8 ) = "12";
{ \ar@{-} " 9"; "10" };
{ \ar@{-} "11"; "12" };
{ \ar@{-} " 9"; "11" };
{ \ar@{-} "10"; "12" };
(  8,  8 ) = "13";
(  8, 16 ) = "14";
( 16,  8 ) = "15";
( 16, 16 ) = "16";
{ \ar@{-} "13"; "14" };
{ \ar@{-} "15"; "16" };
{ \ar@{-} "13"; "15" };
{ \ar@{-} "14"; "16" };
\end{xy}
\quad
\begin{xy}
(  0,  8 ) = " 1";
(  0, 16 ) = " 2";
(  8,  8 ) = " 3";
(  8, 16 ) = " 4";
{ \ar@{-} " 1"; " 2" };
{ \ar@{-} " 3"; " 4" };
{ \ar@{-} " 1"; " 3" };
{ \ar@{-} " 2"; " 4" };
(  0,  0 ) = " 5";
(  0,  8 ) = " 6";
( 16,  0 ) = " 7";
( 16,  8 ) = " 8";
{ \ar@{-} " 5"; " 6" };
{ \ar@{-} " 7"; " 8" };
{ \ar@{-} " 5"; " 7" };
{ \ar@{-} " 6"; " 8" };
(  8,  8 ) = " 9";
(  8, 16 ) = "10";
( 12,  8 ) = "11";
( 12, 16 ) = "12";
{ \ar@{-} " 9"; "10" };
{ \ar@{-} "11"; "12" };
{ \ar@{-} " 9"; "11" };
{ \ar@{-} "10"; "12" };
( 12,  8 ) = "13";
( 12, 16 ) = "14";
( 16,  8 ) = "15";
( 16, 16 ) = "16";
{ \ar@{-} "13"; "14" };
{ \ar@{-} "15"; "16" };
{ \ar@{-} "13"; "15" };
{ \ar@{-} "14"; "16" };
\end{xy}
\\ [1mm] 
& \begin{xy}
(  0,  8 ) = " 1";
(  0, 16 ) = " 2";
(  8,  8 ) = " 3";
(  8, 16 ) = " 4";
{ \ar@{-} " 1"; " 2" };
{ \ar@{-} " 3"; " 4" };
{ \ar@{-} " 1"; " 3" };
{ \ar@{-} " 2"; " 4" };
(  0,  0 ) = " 5";
(  0,  8 ) = " 6";
( 16,  0 ) = " 7";
( 16,  8 ) = " 8";
{ \ar@{-} " 5"; " 6" };
{ \ar@{-} " 7"; " 8" };
{ \ar@{-} " 5"; " 7" };
{ \ar@{-} " 6"; " 8" };
(  8,  8 ) = " 9";
(  8, 12 ) = "10";
( 16,  8 ) = "11";
( 16, 12 ) = "12";
{ \ar@{-} " 9"; "10" };
{ \ar@{-} "11"; "12" };
{ \ar@{-} " 9"; "11" };
{ \ar@{-} "10"; "12" };
(  8, 12 ) = "13";
(  8, 16 ) = "14";
( 16, 12 ) = "15";
( 16, 16 ) = "16";
{ \ar@{-} "13"; "14" };
{ \ar@{-} "15"; "16" };
{ \ar@{-} "13"; "15" };
{ \ar@{-} "14"; "16" };
\end{xy}
\quad
\begin{xy}
(  0, 12 ) = " 1";
(  0, 16 ) = " 2";
(  8, 12 ) = " 3";
(  8, 16 ) = " 4";
{ \ar@{-} " 1"; " 2" };
{ \ar@{-} " 3"; " 4" };
{ \ar@{-} " 1"; " 3" };
{ \ar@{-} " 2"; " 4" };
(  0,  0 ) = " 5";
(  0,  8 ) = " 6";
( 16,  0 ) = " 7";
( 16,  8 ) = " 8";
{ \ar@{-} " 5"; " 6" };
{ \ar@{-} " 7"; " 8" };
{ \ar@{-} " 5"; " 7" };
{ \ar@{-} " 6"; " 8" };
(  0,  8 ) = " 9";
(  0, 12 ) = "10";
( 16,  8 ) = "11";
( 16, 12 ) = "12";
{ \ar@{-} " 9"; "10" };
{ \ar@{-} "11"; "12" };
{ \ar@{-} " 9"; "11" };
{ \ar@{-} "10"; "12" };
(  8, 12 ) = "13";
(  8, 16 ) = "14";
( 16, 12 ) = "15";
( 16, 16 ) = "16";
{ \ar@{-} "13"; "14" };
{ \ar@{-} "15"; "16" };
{ \ar@{-} "13"; "15" };
{ \ar@{-} "14"; "16" };
\end{xy}
\quad
\begin{xy}
(  0,  0 ) = " 1";
(  0,  2 ) = " 2";
( 16,  0 ) = " 3";
( 16,  2 ) = " 4";
{ \ar@{-} " 1"; " 2" };
{ \ar@{-} " 3"; " 4" };
{ \ar@{-} " 1"; " 3" };
{ \ar@{-} " 2"; " 4" };
(  0,  2 ) = " 5";
(  0,  4 ) = " 6";
( 16,  2 ) = " 7";
( 16,  4 ) = " 8";
{ \ar@{-} " 5"; " 6" };
{ \ar@{-} " 7"; " 8" };
{ \ar@{-} " 5"; " 7" };
{ \ar@{-} " 6"; " 8" };
(  0,  4 ) = " 9";
(  0,  8 ) = "10";
( 16,  4 ) = "11";
( 16,  8 ) = "12";
{ \ar@{-} " 9"; "10" };
{ \ar@{-} "11"; "12" };
{ \ar@{-} " 9"; "11" };
{ \ar@{-} "10"; "12" };
(  0,  8 ) = "13";
(  0, 16 ) = "14";
( 16,  8 ) = "15";
( 16, 16 ) = "16";
{ \ar@{-} "13"; "14" };
{ \ar@{-} "15"; "16" };
{ \ar@{-} "13"; "15" };
{ \ar@{-} "14"; "16" };
\end{xy}
\quad
\begin{xy}
(  0,  0 ) = " 1";
(  0,  4 ) = " 2";
( 16,  0 ) = " 3";
( 16,  4 ) = " 4";
{ \ar@{-} " 1"; " 2" };
{ \ar@{-} " 3"; " 4" };
{ \ar@{-} " 1"; " 3" };
{ \ar@{-} " 2"; " 4" };
(  0,  4 ) = " 5";
(  0,  6 ) = " 6";
( 16,  4 ) = " 7";
( 16,  6 ) = " 8";
{ \ar@{-} " 5"; " 6" };
{ \ar@{-} " 7"; " 8" };
{ \ar@{-} " 5"; " 7" };
{ \ar@{-} " 6"; " 8" };
(  0,  6 ) = " 9";
(  0,  8 ) = "10";
( 16,  6 ) = "11";
( 16,  8 ) = "12";
{ \ar@{-} " 9"; "10" };
{ \ar@{-} "11"; "12" };
{ \ar@{-} " 9"; "11" };
{ \ar@{-} "10"; "12" };
(  0,  8 ) = "13";
(  0, 16 ) = "14";
( 16,  8 ) = "15";
( 16, 16 ) = "16";
{ \ar@{-} "13"; "14" };
{ \ar@{-} "15"; "16" };
{ \ar@{-} "13"; "15" };
{ \ar@{-} "14"; "16" };
\end{xy}
\quad
\begin{xy}
(  0,  0 ) = " 1";
(  0,  4 ) = " 2";
( 16,  0 ) = " 3";
( 16,  4 ) = " 4";
{ \ar@{-} " 1"; " 2" };
{ \ar@{-} " 3"; " 4" };
{ \ar@{-} " 1"; " 3" };
{ \ar@{-} " 2"; " 4" };
(  0,  4 ) = " 5";
(  0,  8 ) = " 6";
( 16,  4 ) = " 7";
( 16,  8 ) = " 8";
{ \ar@{-} " 5"; " 6" };
{ \ar@{-} " 7"; " 8" };
{ \ar@{-} " 5"; " 7" };
{ \ar@{-} " 6"; " 8" };
(  0,  8 ) = " 9";
(  0, 12 ) = "10";
( 16,  8 ) = "11";
( 16, 12 ) = "12";
{ \ar@{-} " 9"; "10" };
{ \ar@{-} "11"; "12" };
{ \ar@{-} " 9"; "11" };
{ \ar@{-} "10"; "12" };
(  0, 12 ) = "13";
(  0, 16 ) = "14";
( 16, 12 ) = "15";
( 16, 16 ) = "16";
{ \ar@{-} "13"; "14" };
{ \ar@{-} "15"; "16" };
{ \ar@{-} "13"; "15" };
{ \ar@{-} "14"; "16" };
\end{xy}
\quad
\begin{xy}
(  0,  0 ) = " 1";
(  0,  8 ) = " 2";
( 16,  0 ) = " 3";
( 16,  8 ) = " 4";
{ \ar@{-} " 1"; " 2" };
{ \ar@{-} " 3"; " 4" };
{ \ar@{-} " 1"; " 3" };
{ \ar@{-} " 2"; " 4" };
(  0,  8 ) = " 5";
(  0, 10 ) = " 6";
( 16,  8 ) = " 7";
( 16, 10 ) = " 8";
{ \ar@{-} " 5"; " 6" };
{ \ar@{-} " 7"; " 8" };
{ \ar@{-} " 5"; " 7" };
{ \ar@{-} " 6"; " 8" };
(  0, 10 ) = " 9";
(  0, 12 ) = "10";
( 16, 10 ) = "11";
( 16, 12 ) = "12";
{ \ar@{-} " 9"; "10" };
{ \ar@{-} "11"; "12" };
{ \ar@{-} " 9"; "11" };
{ \ar@{-} "10"; "12" };
(  0, 12 ) = "13";
(  0, 16 ) = "14";
( 16, 12 ) = "15";
( 16, 16 ) = "16";
{ \ar@{-} "13"; "14" };
{ \ar@{-} "15"; "16" };
{ \ar@{-} "13"; "15" };
{ \ar@{-} "14"; "16" };
\end{xy}
\quad
\begin{xy}
(  0,  0 ) = " 1";
(  0,  8 ) = " 2";
( 16,  0 ) = " 3";
( 16,  8 ) = " 4";
{ \ar@{-} " 1"; " 2" };
{ \ar@{-} " 3"; " 4" };
{ \ar@{-} " 1"; " 3" };
{ \ar@{-} " 2"; " 4" };
(  0,  8 ) = " 5";
(  0, 12 ) = " 6";
( 16,  8 ) = " 7";
( 16, 12 ) = " 8";
{ \ar@{-} " 5"; " 6" };
{ \ar@{-} " 7"; " 8" };
{ \ar@{-} " 5"; " 7" };
{ \ar@{-} " 6"; " 8" };
(  0, 12 ) = " 9";
(  0, 14 ) = "10";
( 16, 12 ) = "11";
( 16, 14 ) = "12";
{ \ar@{-} " 9"; "10" };
{ \ar@{-} "11"; "12" };
{ \ar@{-} " 9"; "11" };
{ \ar@{-} "10"; "12" };
(  0, 14 ) = "13";
(  0, 16 ) = "14";
( 16, 14 ) = "15";
( 16, 16 ) = "16";
{ \ar@{-} "13"; "14" };
{ \ar@{-} "15"; "16" };
{ \ar@{-} "13"; "15" };
{ \ar@{-} "14"; "16" };
\end{xy}
\end{array}
$}
\vspace{-2mm}
\caption{Decompositions of the unit square ($d=2)$ using bisections ($p=2$) into $n \le 4$ subrectangles}
\label{degrees1234}
\end{center}
\end{figure}

Observe that in the case of $d = 1$, $p = 2$ and $n \ge 1$, 
Definition~\ref{defndecomp} reduces to subdividing the unit interval into $n$ subintervals 
using bisections. 
Thus, $C_{1,2}(n) = C(n)$ gives the ordinary Catalan numbers. 
More generally, $C_{1,p}(n)$ gives the $p$-ary Catalan numbers \cite{Aval,GKP}:
\[
C_{1,p}\left( n \right) = \frac{1}{n} \binom{ \frac{p(n-1)}{p-1}}{ \frac{n-1}{p-1}}
\quad\quad\quad
( n \ge 1 ).
\]
Thus $C_{d,p}(n)$ is a two-parameter generalization of the Catalan numbers. 

Another generalization of $C_{1,p}(n)$ are the Fuss--Catalan numbers (also known as Raney numbers): Given integers $r \geq 1, p \geq 2$, and $m \geq 0$, the quantity
\[
R_{r,p}(m) = \frac{r}{mp+r} \binom{mp+r}{m}
\]
counts (among other things) the number of plane rooted trees whose root node has degree $r$, with $m$ non-root internal nodes that each has exactly $p$ children~\cite{BeagleyD15}. Notice that when $r=1$, $R_{1,p}(m)$ gives the number of trees whose root node is attached to a single full $p$-ary tree with $n = 1+ m(p-1)$ leaves, and so we see that $R_{1,p}(m) = C_{1,p}(1+m(p-1))$ for every $m \geq 0$ and $p \geq 2$. For several other higher-dimensional generalizations of the Catalan numbers, see \cite{Brin, FreyS, Gould, Kahkeshani, Koshy}.

\subsection{Interchange laws for operations of higher arity}

We provide another interpretation of $C_{d,p}(n)$ using $d$ distinct $p$-ary operations 
(denoted by $d$ operation symbols or by $d$ types of parentheses), 
generalizing interpretation (i) of the ordinary Catalan numbers.

\begin{definition}
Fix integers $d \ge 1$ (\emph{dimension}) and $p \ge 2$ (\emph{arity}).
Let $S$ be a set and fix operations $f_1, \dots, f_d \colon S^p \to S$. 
Let $A = ( a_{ij} )$ be a $p \times p$ array of elements of $S$.
If $f_k, f_{\ell}$ are two of the operations, 
then we may either apply $f_k$ to each row of $A$ and then apply $f_\ell$ to the results,
or apply $f_\ell$ to each column of $A$ and then apply $f_k$ to the results.
If for every array $A$ both ways produce the same element of $S$,
\begin{equation}
\label{intlaw}
\begin{array}{rl}
&f_\ell \big( \,
f_k( a_{11}, \dots, a_{1p} ), \,
f_k( a_{21}, \dots, a_{2p} ), \,
\dots, \,
f_k( a_{p1}, \dots, a_{pp} ) \,
\big)
=
\\[2pt]
&
f_k \big( \,
f_\ell( a_{11}, \dots, a_{p1} ), \,
f_\ell( a_{12}, \dots, a_{p2} ), \,
\dots, \,
f_\ell( a_{1p}, \dots, a_{pp} ) \,
\big),
\end{array}
\end{equation}
then we say that $f_k$ and $f_\ell$ satisfy the \emph{interchange law}. 
If equation \eqref{intlaw} holds for every pair of distinct operations 
then we have an \emph{interchange system} of arity $p$ and dimension $d$.
\end{definition}

In universal algebra (resp.,~algebraic operads), 
interchange systems were introduced in the early 1960s by Evans \cite{E} 
(resp.,~in the early 1970s by Boardman and Vogt \cite{BV}).
In the case of $d = p = 2$, two binary operations satisfying the interchange law 
first appeared in the late 1950s in Godement's \textit{five rules of functorial calculus} \cite[Appendix 1, (V)]{G}.

If we denote the two binary operations by $( - - )$ and $\{ - - \}$,  then
the interchange law can be restated as
\[
\{ \, ( a_{11} \, a_{12} ) \, ( a_{21} \, a_{22} ) \, \}
=
( \, \{ a_{11} \, a_{21} \} \, \{ a_{12} \, a_{22} \} \, ).
\]
Observe that the operations trade places and the arguments $a_{12}$, $a_{21}$ transpose. 
The interchange laws correspond naturally to the hypercube decompositions 
defined in Section~\ref{subsec12}: 
If we regard $( - - )$ and $\{ - - \}$ respectively as vertical and horizontal bisections 
of rectangles in $\mathbb{R}^2$, then the interchange law expresses the equivalence 
of the two ways of partitioning a square into four equal subsquares:
\begin{center}
\scalebox{1.0}{
$
\{ \, ( a_{11} \, a_{12} ) \, ( a_{21} \, a_{22} ) \, \}
\; = 
\begin{array}{c}
\begin{tikzpicture}[draw=black, x=8 mm, y=8 mm]
\node [Square] at ($(0, 0.0)$) {$a_{11}$};
\node [Square] at ($(1, 0.0)$) {$a_{12}$};
\node [Square] at ($(0,-1.2)$) {$a_{21}$};
\node [Square] at ($(1,-1.2)$) {$a_{22}$};
\end{tikzpicture}
\end{array}
=
\begin{array}{c}
\begin{tikzpicture}[draw=black, x=8 mm, y=8 mm]
\node [Square] at ($(0, 0)$) {$a_{11}$};
\node [Square] at ($(1, 0)$) {$a_{12}$};
\node [Square] at ($(0,-1)$) {$a_{21}$};
\node [Square] at ($(1,-1)$) {$a_{22}$};
\end{tikzpicture}
\end{array}
=
\begin{array}{c}
\begin{tikzpicture}[draw=black, x=8 mm, y=8 mm]
\node [Square] at ($(0.0, 0)$) {$a_{11}$};
\node [Square] at ($(1.2, 0)$) {$a_{12}$};
\node [Square] at ($(0.0,-1)$) {$a_{21}$};
\node [Square] at ($(1.2,-1)$) {$a_{22}$};
\end{tikzpicture}
\end{array}
= \;
( \, \{ a_{11} \, a_{21} \} \, \{ a_{12} \, a_{22} \} \, ).
$
}
\end{center}
\noindent
For further information on algebraic operads and higher categories, 
see \cite{BaezDolan,BB,BW,BM,Cheng,Kock,Leinster,LV,May,Street}.

Using the Boardman--Vogt tensor product of operads, the third author and Dotsenko \cite{BD} showed 
that this correspondence between interchange laws and hypercube partitions extends to 
arbitrary arity $p$ and dimension $d$, in which case  $f_i$ ($i = 1, \ldots, d$) 
corresponds to the $H_i$ operation of Definition~\ref{defndecomp} 
(the dissection of a $d$-dimensional subrectangle into $p$ equal parts by hyperplanes 
orthogonal to the $i^{\tn{th}}$ coordinate axis). 
They also proved the following result.

\begin{theorem}[\cite{BD}, \S3.1]\label{BDTheorem}
Define the generating function
\[
y = y_{d,p}(x) = \sum_{n \geq 1} C_{d,p}(n) x^n.
\]
Then $y$ satisfies this polynomial functional equation:
\begin{equation}
\label{dgfeq}
\sum_{k=0}^d (-1)^k \binom{d}{k} y^{p^k} = x.
\end{equation}
\end{theorem}

For $d = 1$ and $p = 2$, equation \eqref{dgfeq} reduces to equation \eqref{CatalanFnEq}, 
the functional equation for the ordinary Catalan numbers. 

One may also regard $C_{d, p}(n)$ as the number of association types \cite{Eth}
(or placements of parentheses and operation symbols) of degree $n$ in higher-dimensional algebra
\cite{Brown,BHS,BP} with $d$ operations of arity $p$.
Similar (but inequivalent) constructions in the literature include guillotine partitions, VLSI floorplans, and planar rectangulations \cite{ABP,ABPR,ABBMP,ABMP,AM,CM,BDHe,LR,R}.

\subsection{Outline of this paper}

In Section~\ref{sec2}, we use Lagrange inversion to prove a simple closed formula 
(a finite sum) for $C_{d,p}(n)$ and then consider two special cases.
In Section~\ref{sec3}, we use analytic methods to determine the asymptotic behavior 
of $C_{d,p}(n)$.
In Section~\ref{sec4}, we provide an alternative proof to Theorem~\ref{BDTheorem} that is purely combinatorial and does not involve homological algebra. We shall also show that $C_{d,p}(n)$ counts a restricted set of
$p$-ary trees by establishing a bijection between these trees and hypercube decompositions, hence generalizing interpretation (ii) of the ordinary Catalan numbers. In Section~\ref{sec5}, we indicate some directions for further research, and explain how our results may be understood from the point of view of algebraic operads.

\section{Enumeration formulas}\label{sec2}

We first derive a summation formula for $C_{d,p}(n)$ and
then discuss special cases for small values of $d$ and $p$.
We use the functional equation~\eqref{dgfeq} to obtain our closed formula.

\begin{proposition}
\label{binarytheorem}
For all integers $d \geq 1$, $p \geq 2$, and $n \geq 1$ we have
\[
C_{d,p}(n)
=
\frac{1}{n}
\sum_{t_1, \ldots, t_d}
\left[
\binom{ n {-} 1 {+} t_1 {+} \cdots {+} t_d }{ \, n{-}1, \, t_1, \, t_2, \, \ldots, \, t_d \, }
\prod_{k=1}^d \left((-1)^{t_k(k+1)} \binom{d}{k}^{t_k}  \right)
\right],
\]
where the sum is over all integers $t_1, \ldots, t_d \ge 0$ such that
\[
\sum_{k=1}^d t_k ( p^k{-}1 ) = n{-}1.
\]
\end{proposition}

\begin{proof}
We may rearrange the functional equation \eqref{dgfeq} to obtain $y = x \phi(y)$ where
\[
\phi(y) = \left( \sum_{k=0}^d (-1)^k \binom{d}{k} y^{p^k-1} \right)^{\!\!-1}.
\]
Since $\phi(y)$ can be expanded as a formal power series in $y$ with nonzero constant term, 
we apply Lagrange inversion \cite{Gessel} to obtain
\[
[x^n] y
=
\frac{1}{n} [y^{n-1}] \big( \phi(y)^n \big)
=
\frac{1}{n} [y^{n-1}] \left( \sum_{k=0}^d (-1)^k \binom{d}{k} y^{p^k -1} \right)^{\!\!-n}
\!\!\!\!\!,
\]
where $[x^n] y$ denotes the coefficient of $x^n$ in the power series $y$.
If we expand the factor with the negative exponent and simplify the result then we obtain
\begin{align*}
\phi(y)^n
&=
\sum_{j \ge 0}
\binom{n{-}1{+}j}{j}
\left( \sum_{k=1}^d (-1)^{k+1} \binom{d}{k} y^{p^k -1} \right)^{j}
\\
&=
\sum_{j \geq 0}
\binom{n{-}1{+}j}{j}
\sum_{\substack{ t_1, \ldots, t_d \geq 0 \\ t_1+\cdots+t_d=j}}
\left[
\binom{j}{t_1, t_2, \ldots, t_d}
\prod_{k=1}^d
\left( (-1)^{t_k(k+1)} \binom{d}{k}^{t_k}  y^{t_{k} (p^k -1)} \right)
\right]
\\
&=
\sum_{ t_1, \ldots, t_d \geq 0 }
\left[
\binom{n{-}1+ \sum_{i=1}^d t_i }{n{-}1, t_1, t_2, \ldots, t_d}
\prod_{k=1}^d
\left( (-1)^{t_k(k+1)} \binom{d}{k}^{t_k}  y^{t_{k} (p^k -1)} \right)
\right].
\end{align*}

In the last step we used the equation $\sum_{k=1}^d t_k = j$ to eliminate $j$, 
and then used the obvious combinatorial identity
\[
\binom{n{-}1{+}j}{j} \binom{j}{t_1,\ldots,t_d} = \binom{n{-}1{+}j}{n{-}1,t_1,\ldots,t_d}.
\]
The term of interest $y^{n-1}$ occurs if and only if
\[
\sum_{k=1}^d t_k (p^k{-}1) = n{-}1,
\]
which completes the proof.
\end{proof}

Proposition \ref{binarytheorem} implies simple closed formulas for $C_{d,p}(n)$ 
for small $d$ and $p$. For instance, in the case of $d=p=2$, we obtain the following.

\begin{corollary}
\label{cor22}
The number of dyadic partitions of the unit square into $n$ rectangles is
\[
C_{2,2}(n)
=
\frac{1}{n}
\sum_{i=0}^{ \lfloor \frac{n-1}{3} \rfloor }
\binom{ 2(n{-}1{-}i) }{ \, n{-}1, \, n{-}1{-}3i, \, i \, } (-1)^{i} \, 2^{n-1-3i}.
\]
\end{corollary}

\begin{proof}
For $d=p=2$, Proposition~\ref{binarytheorem} gives
\begin{equation}
\label{d2p2}
C_{2,2}(n) = \frac{1}{n}
\sum_{ \substack{ t_1, \, t_2 \geq 0 \\[1pt] t_1 + 3t_2 = n-1 } }
\binom{ n{-}1{+}t_1{+}t_2 }{ \, n{-}1, \, t_1, \, t_2 \, }
(-1)^{t_2} \, 2^{t_1}.
\end{equation}
If we set $t_2 = i$ then $t_1 = n{-}1{-}3i$, and the sum \eqref{d2p2} is empty when
$n{-}1{-}3i < 0$, or equivalently $i > \frac{n-1}{3}$. Hence \eqref{d2p2} specializes to the stated result.
\end{proof}

This gives what we believe to be the first explicit non-recursive formula for sequence \seqnum{A236339}:
\[
1, \; 2, \; 8, \; 39, \; 212, \; 1232, \; 
7492, \; 47082, \; 303336, \; 1992826, \; 13299624, \; 89912992, \; \dots.
\]
Likewise, for the case of $d=3, p=2$, Proposition~\ref{binarytheorem} implies the following.

\begin{corollary}
\label{cor32}
The number of dyadic partitions of the unit cube into $n$ subrectangles
is
\[
C_{3,2}(n) =
\frac{1}{n}
\sum_{j=0}^{\lfloor \frac{n{-}1}{7} \rfloor}
\sum_{i=0}^{\lfloor \frac{n{-}1{-}7j}{3} \rfloor}
\binom{ 2(n{-}1{-}i{-}3j) }{ n{-}1, n{-}1{-}3i{-}7j, i, j}
(-1)^{i} \, 3^{n-1-2i-7j}.
\]
\end{corollary}

\begin{proof}
Similar to the proof of Corollary \ref{cor22}.
For $d=3$, $p=2$ Proposition~\ref{binarytheorem} gives
\begin{equation}
\label{d3p2}
C_{3,2}(n) =  \frac{1}{n}
\sum_{ \substack{ t_1, t_2,t_3 \geq 0 \\[1pt] t_1 + 3t_2 +7t_3 = n-1}}
\binom{ n{-}1{+}t_1{+}t_2{+}t_3 }{ \, n{-}1, \, t_1, \, t_2, \, t_3 } (-1)^{t_2} \, 3^{t_1+t_2}.
\end{equation}
If we set $t_2 = i$, $t_3 = j$ then $t_1 = n{-}1{-}3i{-}7j$, and the sum \eqref{d3p2}
is empty for $n{-}1{-}3i{-}7j < 0$.
Hence we may restrict the summation indices to
$0 \leq j \leq \frac{n-1}{7}$ and $0 \leq i \leq \frac{n-1-7j}{3}$,
and so \eqref{d3p2} simplifies to the stated result.
\end{proof}

This gives an explicit non-recursive formula for the sequence \seqnum{A236342}:
\[
1, \; 3, \; 18, \; 132, \; 1080, \; 9450, \; 
86544, \; 819154, \; 7949532, \; 78671736, \; 790930728, \; \dots.
\]
We remark that, while Proposition~\ref{binarytheorem} gives a simple and compact summation formula for $C_{d,p}(n)$, the summation contains $O(n^{d-1})$ terms for fixed $d$ and $p$. For a more computationally efficient method to obtain $C_{d,p}(n)$, one can use the fact that the generating function $y$ is differentiably-finite (since it satisfies~\eqref{dgfeq} and is thus an algebraic function), and generate a linear recurrence for the coefficients of $y$ based on that. For the details of this procedure, see~\cite[Chapter 6]{Stanley99}, ~\cite{BanderierD13}, and the references therein.

\section{Asymptotic behavior and growth rate}
\label{sec3}

We now consider the asymptotic behavior and growth rate of $C_{d,p}(n)$. 
Recall again the functional equation~\eqref{dgfeq}.
If we define the polynomial
\[
q_{d,p}(z) = \sum_{k=0}^d (-1)^k \binom{d}{k} z^{p^k},
\]
then~\eqref{dgfeq} can be restated simply as $q_{d,p}(y) = x$. 
We call this polynomial $q(z)$ when $d, p$ are clear from the context. 
Figure~\ref{figqz} shows the graph of $q(z)$ for small values of $d$ and $p$. 

\begin{figure}[ht]
\begin{center}
\begin{tabular}{cc}
\begin{tikzpicture}[scale =3, xscale=0.75,yscale=1, font = \scriptsize]
\def\xlb{-1}; \def\xub{1}; \def\ylb{-0.4};  \def\yub{0.4}; \def\buf{0.2};
\draw [->] (\xlb - \buf,0) -- (\xub + \buf,0);
\draw [->] (0, \ylb - \buf) -- (0, \yub + \buf);
\foreach \x in {\xlb ,...,\xub}
{
\ifthenelse{\NOT 0 = \x}{\draw[thick](\x ,-1pt) -- (\x ,1pt);}{}
\ifthenelse{\NOT 0 = \x}{\node[anchor=north] at (\x,0) (label) {{ $\x$}};}{}
}
\foreach \y in {-0.5, 0.5}
{
\draw[thick](-1pt,\y) -- (1pt,\y);
\node[anchor=east] at (0,\y) (label) {{ $\y$}};
}
\node[anchor=north east] at (0,0) (label3) {$0$};
\node[anchor=west] at (\xub + \buf,0) (label3) {$z$};
\node[anchor=south] at (0, \yub+\buf) (label3) {$q(z)$};
\draw[<->, thick, domain= -0.35 : 1.2, samples = 78] plot (\x, {\x - 2*\x * \x + \x*\x*\x*\x} );
\end{tikzpicture}

&\quad\quad\quad\quad

\begin{tikzpicture}[scale =3, xscale=0.75,yscale=1, font = \scriptsize]
\def\xlb{-1}; \def\xub{1}; \def\ylb{-0.4};  \def\yub{0.4}; \def\buf{0.2};
\draw [->] (\xlb - \buf,0) -- (\xub + \buf,0);
\draw [->] (0, \ylb - \buf) -- (0, \yub + \buf);
\foreach \x in {\xlb ,...,\xub}
{
\ifthenelse{\NOT 0 = \x}{\draw[thick](\x ,-1pt) -- (\x ,1pt);}{}
\ifthenelse{\NOT 0 = \x}{\node[anchor=north] at (\x,0) (label) {{ $\x$}};}{}
}
\foreach \y in {-0.5, 0.5}
{
\draw[thick](-1pt,\y) -- (1pt,\y);
\node[anchor=east] at (0,\y) (label) {{ $\y$}};
}
\node[anchor=north east] at (0,0) (label3) {$0$};
\node[anchor=west] at (\xub + \buf,0) (label3) {$z$};
\node[anchor=south] at (0, \yub+\buf) (label3) {$q(z)$};
\draw[<->, thick, domain= -0.32 : 1.15, samples = 74] plot (\x, {\x - 3*\x * \x + 3*\x*\x*\x*\x - \x*\x*\x*\x*\x*\x*\x*\x } );
\end{tikzpicture}

\\
$q_{2,2}(z) = z {-} 2 z^2 {+} z^4$ 
&\qquad 
$q_{3,2}(z) = z {-} 3 z^2 {+} 3 z^4 {-} z^8$ 
\\[3mm]

\begin{tikzpicture}[scale =3, xscale=0.75,yscale=1, font = \scriptsize]
\def\xlb{-1}; \def\xub{1}; \def\ylb{-0.4};  \def\yub{0.4}; \def\buf{0.2};
\draw [->] (\xlb - \buf,0) -- (\xub + \buf,0);
\draw [->] (0, \ylb - \buf) -- (0, \yub + \buf);
\foreach \x in {\xlb ,...,\xub}
{
\ifthenelse{\NOT 0 = \x}{\draw[thick](\x ,-1pt) -- (\x ,1pt);}{}
\ifthenelse{\NOT 0 = \x}{\node[anchor=north] at (\x,0) (label) {{ $\x$}};}{}
}
\foreach \y in {-0.5, 0.5}
{
\draw[thick](-1pt,\y) -- (1pt,\y);
\node[anchor=east] at (0,\y) (label) {{ $\y$}};
}
\node[anchor=north east] at (0,0) (label3) {$0$};
\node[anchor=west] at (\xub + \buf,0) (label3) {$z$};
\node[anchor=south] at (0, \yub+\buf) (label3) {$q(z)$};
\draw[<->, thick, domain= -1.08 : 1.08, samples = 109] plot (\x, {\x - 2*\x * \x* \x + \x*\x*\x*\x* \x* \x* \x* \x* \x} );
\end{tikzpicture}

&\quad\quad\quad\quad

\begin{tikzpicture}[scale =3, xscale=0.75,yscale=1, font = \scriptsize]
\def\xlb{-1}; \def\xub{1}; \def\ylb{-0.4};  \def\yub{0.4}; \def\buf{0.2};
\draw [->] (\xlb - \buf,0) -- (\xub + \buf,0);
\draw [->] (0, \ylb - \buf) -- (0, \yub + \buf);
\foreach \x in {\xlb ,...,\xub}
{
\ifthenelse{\NOT 0 = \x}{\draw[thick](\x ,-1pt) -- (\x ,1pt);}{}
\ifthenelse{\NOT 0 = \x}{\node[anchor=north] at (\x,0) (label) {{ $\x$}};}{}
}
\foreach \y in {-0.5, 0.5}
{
\draw[thick](-1pt,\y) -- (1pt,\y);
\node[anchor=east] at (0,\y) (label) {{ $\y$}};
}
\node[anchor=north east] at (0,0) (label3) {$0$};
\node[anchor=west] at (\xub + \buf,0) (label3) {$z$};
\node[anchor=south] at (0, \yub+\buf) (label3) {$q(z)$};
\draw[<->, thick, domain= -1.03 : 1.03, samples = 109] 
plot (\x, {\x - 3*\x*\x*\x + 3*\x*\x*\x*\x*\x*\x*\x*\x*\x - \x*\x*\x*\x*\x*\x*\x*\x*\x*\x*\x*\x*\x*\x*\x*\x*\x*\x*\x*\x*\x*\x*\x*\x*\x*\x*\x} );
\end{tikzpicture}
\\
$q_{2,3}(z) = z {-} 2 z^3 {+} z^9$
&\qquad
$q_{3,3}(z) = z {-} 3 z^3 {+} 3 z^9 {-} z^{27}$
\end{tabular}
\end{center}
\vspace{-5mm}
\caption{Graphs of the polynomials $q_{d,p}(z)$ for $(d,p) = (2,2), (3,2), (2,3), (3,3)$}
\label{figqz}
\end{figure}

\subsection{Asymptotic behavior}

For $p=d=2$, Kot\v{e}\v{s}ovec gave the following asymptotic formula
(OEIS, sequence \seqnum{A236339}): 
\[
C_{2,2}(n) \,\sim\, \frac{1}{\sqrt{ -2\pi q''(s)} } \, n^{-3/2} q(s)^{1/2-n}.
\]
Here $q(y) = y^4-2y^2+y$, 
hence $q'(y) = 4y^3 - 4y + 1$ and $q''(y) = 12y^2 - 4$, and $s$ is the smallest positive real number where $q'(s) = 0$. 
Kot\v{e}\v{s}ovec's proof~\cite{Kotesovec18} relies on a theorem of Bender~\cite{Bender} 
that was later corrected and refined; see~\cite[Theorem VII.3, p.~468]{FS2}.
We provide an asymptotic formula for $C_{d,p}(n)$ for all $d \geq 1$ and $p \geq 2$. An important property of the power series for $C_{d,p}(n)$ is its periodicity:
 
\begin{definition}
A power series $\phi$ in the variable $z$ is \emph{$k$-periodic} 
for some positive integer $k$ if $k$ is maximal subject to the condition that
there exists a unique $\ell \in \set{0,1 , \ldots, k{-}1}$ such that $[z^n] \phi(z) = 0$ 
for all $n \not\equiv \ell$ (mod $k$). 
\end{definition}

Then we have the following:

\begin{lemma}
The power series $y = \displaystyle\sum_{n \geq 1} C_{d,p}(n)x^n$ is $(p{-}1)$-periodic.
\end{lemma}

\begin{proof}
We have $C_{d,p}(1)= 1$ for all $d$ and $p$, and each subsequent $p$-ary partition 
of a subrectangle increases the number of regions by $p{-}1$. 
Hence $C_{d,p}(n) = 0$ if $n \not\equiv 1$ (mod $p{-}1$). 
\end{proof}

The following result is a combination of \cite[Theorem VI.6, p.~404]{FS2} and \cite[Note VI.17, p.~407]{FS2}:

\begin{theorem}\label{FSasymp}
Let $y$ be a power series in $x$.
Let $\phi\colon \mathbb{C} \to \mathbb{C}$ be a function with the following properties:
\begin{enumerate}
\item[(i)]
$\phi$ is analytic at $z = 0$ and $\phi(0) > 0$;
\item[(ii)]
$y = x \phi(y)$;
\item[(iii)]
$[z^n] \phi(z) \geq 0$ for all $n \geq 0$, and $[z^n] \phi(z) \neq 0$ for some $n \geq 2$.
\item[(iv)]
There exists a (then necessarily unique) real number $s \in (0,r)$ such that
$\phi(s) = s \phi'(s)$, where $r$ is the radius of convergence of $\phi$.
\item[(v)]
The power series of $\phi$ is $k$-periodic.
\end{enumerate}
Then for $n \equiv 1$ \tn{(mod $k$)}, 
\[
[x^n] y \,\sim\, 
k \, \sqrt{ \frac{\phi(s)}{2\pi  \phi''(s)}} \, n^{-3/2} \left(  \phi'(s) \right)^n.
\]
\end{theorem}

Using Theorem~\ref{FSasymp}, we obtain the following:

\begin{theorem}\label{Cdpnasymp}
For all integers $d \geq 1$ and $p \geq 2$, define
\[
q(z) = \sum_{k=0}^d (-1)^k \binom{d}{k} z^{p^k}.
\] 
Let $s > 0$ be the smallest real number such that $q'(s) = 0$. 
Then for $n \equiv 1$ \tn{(mod $p{-}1$)},
\[
C_{d,p}(n) \,\sim\, \frac{p{-}1}{ \sqrt{ -2\pi q''(s)}} \, n^{-3/2} \, q(s)^{\tfrac12-n}.
\]
\end{theorem}

\begin{proof}
Since $q(y) = x$, we have $y=x\phi(y)$ where 
\begin{equation}\label{Phi0}
\phi(z) = \frac{z}{q(z)} = \frac{1}{\sum_{k=0}^d (-1)^k \binom{d}{k} z^{p^k-1}}.
\end{equation}
We now verify the analytic conditions listed in Theorem~\ref{FSasymp}. First, we show that $\phi(z)$ satisfies condition $(v)$ for $k=p-1$. 
The power series of $\phi$ has the form
\begin{equation}\label{Phi1}
\phi(z) 
= 
\sum_{n \geq 0} a_n z^n
=
\frac{1}{\sum_{k=0}^d (-1)^k \binom{d}{k} z^{p^k-1}} 
= 
\sum_{j \geq 0} \left( \sum_{k=1}^d (-1)^{k-1} \binom{d}{k} z^{p^k-1}  \right)^j.
\end{equation}
Since $p{-}1 \mid p^k{-}1$ for all $k \geq 1$, 
we see that $a_n = 0$ when $p{-}1 \nmid n$. Also, notice that $a_0=1$ and $a_{p-1}=d$, so the periodicity of $p-1$ is indeed maximal.

Second, we show that $\phi(z)$ satisfies $(iii)$. It is shown above that $a_n \neq 0$ only when $p-1 | n$, so we will focus on $n$'s that are multiples of $p-1$. We prove by induction on $n$ that for all $d \geq 1$, $p \geq 2$, and $n \geq p-1$ we have
\[
a_n \geq (d-1) a_{n-(p-1)}.
\]
For the basis, from \eqref{Phi1} we see that $a_n = d^{n/(p-1)}$ for all $n < p^2 -1$.
For the inductive step, from \eqref{Phi0} we see that $a_n$ satisfies the recurrence relation
\begin{align*}
a_n &= \sum_{k=1}^d (-1)^{k-1} \binom{d}{k} a_{n - (p^k-1)} 
\\
&= (d{-}1)a_{n-(p-1)} + \underbrace{\left( a_{n-(p-1)} - \binom{d}{2} a_{n-(p^2-1)} \right)}_{\tn{(I)}} + \underbrace{\sum_{k=3}^d (-1)^{k-1} \binom{d}{k} a_{n - (p^k-1)}}_{\tn{(II)}}. 
\end{align*}
To show that $a_n \geq (d-1)a_{n-(p-1)}$, we verify that the terms (I) and (II) are both nonnegative.
By the inductive hypothesis, for all values of $d$ and $p$ we have
\[
a_{n-(p-1)} 
\geq 
(d{-}1)^{p} a_{n-(p-1)-p(p-1)}  
= 
(d{-}1)^p a_{n-(p^2-1)}  
\geq 
\binom{d}{2} a_{n-(p^2-1)},
\]
and so (I) is nonnegative. Furthermore, for $k \geq 3$ we have
\[
\binom{d}{k} a_{n-(p^k-1)} 
\geq 
\binom{d}{k} (d{-}1)^{p^k} a_{n-(p^{k+1}-1)} \geq 
\binom{d}{k} (d{-}1) a_{n-(p^{k+1}-1)} \geq 
\binom{d}{k{+}1} a_{n-(p^{k+1}-1)}.
\]
The first inequality follows from the inductive hypothesis, and the third from the fact that
\[
\binom{d}{k}(d{-}1) \geq \binom{d}{k{+}1}
\] 
for all $d$, $k$. Thus, the term (II) is also nonnegative. Since $a_n \geq (d-1)a_{n-(p-1)}$, we hence conclude that $a_n \geq 0$ for all $n \geq 1$.

We now consider condition $(iv)$. 
Clearly, 
\begin{equation}\label{Phi2}
\phi'(z) = \frac{q(z) - zq'(z)}{q(z)^2},
\quad\quad\quad
\phi''(z) = \frac{-zq(z)q''(z) - 2q(z)q'(z) + 2z(q'(z))^2}{q(z)^3}.
\end{equation}
Let $r$ be the radius of convergence of $\phi$ at $z = 0$. 
Since $\phi(z) = z/q(z)$, we see that $r$ is the smallest positive solution to $q(r) = 0$. 
We show there exists $s \in (0,r)$ with $\phi(s) = s\phi'(s)$. 
Notice that
\[
\phi(s) = s \phi'(s) 
\quad\implies\quad 
\frac{s}{q(s)} = s \left( \frac{q(s) - sq'(s)}{q(s)^2} \right) 
\quad\implies\quad 
sq'(s) = 0.
\]
Since $q(0) = q(r)=0$ and $q$ is differentiable, it follows that $q'(s) =0$ 
for some $s \in (0,r)$. 

Now that the analytic assumptions on $\phi(z)$ have been verified, 
we may establish the asymptotic formula. 
Since $q'(s) = 0$, the expressions in \eqref{Phi2} simplifies to
\[
\phi'(s) = \frac{1}{q(s)}, 
\quad\quad\quad 
\phi''(s) = \frac{-sq''(s)}{q(s)^2}.
\]
Therefore, when $n \equiv 1$ (mod $p{-}1$) we have
\begin{align*}
C_{d,p}(n) = [x^n] y 
&\;\sim\;
(p{-}1) \, \sqrt{ \frac{\phi(s)}{2 \pi \phi''(s)} } \; n^{-3/2} \, \phi'(s)^n 
\\
&\;=\; 
(p{-}1) \, 
\sqrt{ 
\frac{ 
s / q(s)
}
{
2 \pi \big( -s q''(s) / q(s)^2 \big)
}
} 
\; n^{-3/2}  \left( \frac{1}{q(s)} \right)^n 
\\
&\;=\; 
\frac{p{-}1}{ \sqrt{ -2\pi q''(s)  }} \, n^{-3/2} \, q(s)^{1/2-n},
\end{align*}
and this completes the proof.
\end{proof}

\subsection{Growth rate}

We define the growth rate of $C_{d,p}(n)$ as
\[
\mathcal{G}_{d,p} = 
\lim_{m \to \infty} 
\frac{ C_{d,p} \big( (m{+}1)(p{-}1)+1) }{ C_{d,p} \big( m(p{-}1)+1 \big) }.
\]

Then we have the following:

\begin{corollary}\label{asym1}
Given fixed integers $d \geq 1$ and $p\geq 2$,
\[
\mathcal{G}_{d,p} = \frac{1}{(q(s))^{p-1}},
\]
where $s$ is the smallest positive real number where $q'(s) =0$.
\end{corollary}

\begin{proof}
This follows immediately from Theorem \ref{Cdpnasymp}.
\end{proof}

We computed $\mathcal{G}_{d,p}$ for various $d$ and $p$; see Figure \ref{figureGdp}.
For $d=1$ (the familiar $p$-ary Catalan numbers) we have
\[
\mathcal{G}_{1,p} = \frac{p^p}{(p{-}1)^{p-1}},
\]
which by Stirling's formula grows almost linearly in $p$ since
$\mathcal{G}_{1,p+1} - \mathcal{G}_{1,p} \approx e$
for $p \geq 1$.
Figure~\ref{figureGdp} suggests that $\mathcal{G}_{d,p}$ also grows almost linearly in $d$
for $d \geq 1$. In fact $\mathcal{G}_{d,p} \approx d \, \mathcal{G}_{1,p}$ for all values of $d,p$ we checked.

\begin{figure}[htb]
\begin{center}
\begin{tikzpicture}[scale = 0.5, xscale = 1.2, yscale = 0.13, font = \small, word node/.style={font=\small}]]
\def\d{12};
\def\xlb{0};
\def\xub{\d};
\def\ylb{0};
\def\yub{60};
\def\xbuf{0.5};
\def\ybuf{5};
\draw [->] (0,0) -- (\xub + \xbuf,0);
\draw [->] (0,0) -- (0, \yub + \ybuf);
\foreach \x in {\xlb ,1, ...,\xub}
{
\ifthenelse{\NOT 0 = \x}{\draw[thick](\x ,-20pt) -- (\x ,20pt);}{}
\ifthenelse{\NOT 0 = \x}{\node[anchor=north] at (\x,0) (label) {{$\x$}};}{}
}
\foreach \y in {10, 20,30,40,50,60}
{
\draw[thick](-2pt, \y ) -- (2pt, \y);
\node[anchor=east] at (0,\y) (label) {{ $\y$}};
}
\node[anchor=north east] at (0,0) (label3) { $0$};
\node[anchor=west] at (\xub + \xbuf,0) (label3) {$d$};
\node[anchor=south] at (0, \yub+\ybuf) (label3) {$\footnotesize \mathcal{G}_{d,p} $};
\draw[thick](
   1,   4.0000)--(
   2, 7.7211)--(
   3, 11.6335)--(
   4, 15.5945)--(
   5, 19.5729)--(
   6, 23.5592)--(
   7, 27.5498)--(
   8, 31.5430)--(
   9, 35.5378)--(
  10, 39.5338)--(
  11, 43.5305)--(
  12, 47.5278
);
\def\y{0.7};
\foreach \position in {(
   1, 4.0000),(
   2, 7.7211),(
   3, 11.6335),(
   4, 15.5945),(
   5, 19.5729),(
   6, 23.5592),(
   7, 27.5498),(
   8, 31.5430),(
   9, 35.5378),(
  10, 39.5338),(
  11, 43.5305),(
  12, 47.5278
)}
{
\node[draw, circle, inner sep=0pt, minimum size = 0.1cm, fill] at \position {};
}
\node[anchor = south] at (13,   47.5278) {$p=2$};
\draw[thick](
1,    6.7500)--(
2,   13.4683)--(
3,   20.2220)--(
4,   26.9764)--(
5,   33.7299)--(
6,   40.4826)--(
7,   47.2347)--(
8,   53.9863
);
\def\y{0.7};
\foreach \position in {(
1,    6.7500),(
2,   13.4683),(
3,   20.2220),(
4,   26.9764),(
5,   33.7299),(
6,   40.4826),(
7,   47.2347),(
8,   53.9863
)}
{
\node[draw, circle, inner sep=0pt, minimum size = 0.1cm, fill] at \position {};
}
\node[anchor = south] at (9,   53.9863) {$p=3$};
\draw[thick](
1,    9.4815)--(
2,   18.9606)--(
3,   28.4431)--(
4,   37.9251)--(
5,   47.4068)--(
6,   56.8885
);
\def\y{0.7};
\foreach \position in {(
1,    9.4815),(
2,   18.9606),(
3,   28.4431),(
4,   37.9251),(
5,   47.4068),(
6,   56.8885
)}
{
\node[draw, circle, inner sep=0pt, minimum size = 0.1cm, fill] at \position {};
}
\node[anchor = south] at (7,   56.8885) {$p=4$};
\draw[thick](
1,     12.207)--(
2,   24.414)--(
3,   36.621)--(
4,   48.828)--(
5,   61.035
);
\def\y{0.7};
\foreach \position in {(
1,     12.207),(
2,   24.414),(
3,   36.621),(
4,   48.828),(
5,   61.035)}
{
\node[draw, circle, inner sep=0pt, minimum size = 0.1cm, fill] at \position {};
}
\node[anchor = south] at (6,   61.035) {$p=5$};
\draw[thick](
1,  14.930)--(
2,   29.860)--(
3,   44.790)--(
4,   59.720
);
\def\y{0.7};
\foreach \position in {(
1,  14.930),(
2,   29.860),(
3,   44.790),(
4,   59.720
)}
{
\node[draw, circle, inner sep=0pt, minimum size = 0.1cm, fill] at \position {};
}
\node[anchor = south] at (4,   59.720) {$p=6$};
\end{tikzpicture}
\end{center}
\vspace{-8mm}
\caption{The growth rate $\mathcal{G}_{d,p}$ for various $d$ and $p$}
\label{figureGdp}
\end{figure}

\section{Interpretation of $C_{d,p}(n)$ in terms of $p$-ary trees}\label{sec4}

Recall the three interpretations of the binary Catalan numbers from Section~\ref{subsec11}:
(i) placements of parentheses, 
(ii) binary trees, and 
(iii) bisections of the unit interval. 
For the numbers $C_{d,p}(n)$, we saw in the previous sections that 
(i) generalizes to the number of ways to apply $d$ distinct $p$-ary operations 
to $n$ arguments while satisfying the interchange laws, 
and (iii) generalizes to the number of ways to divide the $d$-dimensional hypercube 
into $n$ rectangular regions using $p$-ary partitions. 
In this section, we generalize interpretation (ii), 
and provide a combinatorial description of $C_{d,p}(n)$ in terms of certain $p$-ary trees 
by establishing a bijection between these trees and the set of $(d,p,n)$-decompositions.

\subsection{A combinatorial proof of Theorem~\ref{BDTheorem}}

In~\cite{BD}, the proof of Theorem~\ref{BDTheorem} mainly involves ideas from homological algebra. Herein, we give an alternative proof to their result that is purely combinatorial (and, in our opinion, more elementary). The ideas used in our proof will also be helpful later in this section when we discuss the correspondence between hypercube decompositions and trees. 

Given an integer $k >0$, we will let $[k]$ denote the set $\set{1, 2, \ldots, k}$ for convenience. Now, given integers $d \geq 1, p \geq 2$ and $n \geq 1$, let $\D_{d,p,n}$ denote the set of all $(d,p,n)$-decompositions, and let $\D_{d,p} = \bigcup_{n \geq 1} \D_{d,p,n}$. Next, given a set of indices $S \subseteq [d]$, let $D_S \in \D_{d,p}$ denote the decomposition obtained recursively as follows:
\begin{itemize}
\item
$D_{\emptyset}$ is the trivial decomposition $\set{(0,1)^d}$;
\item
Given $S \neq \emptyset$, let $i$ be any index in $S$, and define
\[
D_S := \bigcup_{R \in D_{ S \setminus \set{i}}} H_i(R).
\]
\end{itemize}

\begin{figure}[htb]
\begin{center}
$
\begin{array}[t]{ccc}
\begin{tikzpicture}
[scale=0.2, thick]
\draw(0,0) -- (12,0) -- (12,12) -- (0,12) -- (0,0);
\draw(6,0)--(6,12);
\end{tikzpicture}
~~~&~~~
\begin{tikzpicture}
[scale=0.2, thick]
\draw(0,0) -- (12,0) -- (12,12) -- (0,12) -- (0,0);
\draw(0,6)--(12,6);
\end{tikzpicture}
~~~&~~~
\begin{tikzpicture}
[scale=0.2, thick]
\draw(0,0) -- (12,0) -- (12,12) -- (0,12) -- (0,0);
\draw(6,0)--(6,12);
\draw(0,6)--(12,6);
\end{tikzpicture}
\\
D_{ \set{1}} &D_{ \set{2}} &D_{ \set{1,2}}\\
\end{array}
$
\caption{Illustrating the definition of $D_S$ in the case of $d=p=2$}\label{figDS}
\end{center}
\end{figure}

Notice that $D_{S}$ has $p^{|S|}$ regions of identical volume. Next, given decompositions $D,D' \in \D_{d,p}$, we say that $D'$ \emph{refines} $D$ --- denoted $D' \succeq D$ --- if $D'$ can be obtained from $D$ by a (possibly empty) sequence of $p$-splitting operations. Then, given $S \subseteq [d]$ and integer $n \geq 1$, define
\[
\mathcal{H}_{S,n} := \set{ D \in  \D_{d,p,n}, D \succeq D_S}.
\]
For instance, notice that in the case of $d=p=2$, if $D \in \D_{2,2,n}$ where $n \geq 2$, then either $D \succeq D_{\set{1}}$ or $D \succeq D_{\set{2}}$ (or both), as $D$ is a non-trivial decomposition that involves at least one $2$-splitting operation. Thus, in this case, we see that $\D_{2,2,n} = \mathcal{H}_{\set{1}, n} \cup \mathcal{H}_{\set{2}, n}$ for every $n \geq 2$. Moreover, it is easy to see that $\mathcal{H}_{\set{1},n} \cap \mathcal{H}_{\set{2},n} = \mathcal{H}_{\set{1,2},n}$. Hence, we obtain that, for every $n \geq 2$,
\[
C_{2,2}(n) = |\D_{2,2,n}| = \left| \mathcal{H}_{\set{1}, n} \right| +  \left| \mathcal{H}_{\set{2}, n} \right| -  \left| \mathcal{H}_{\set{1,2}, n} \right|.
\]
We shall see below that the above readily extends to an inclusion-exclusion argument that applies for all $d$ and $p$.

Next, given a decomposition $D$ that refines $D_S$, we would like to establish a correspondence between $D$ and a $p^{|S|}$-tuple of decompositions. To do that, it will be useful to define the notion of scaling. Given (rectangular) regions $R, R' \subseteq (0,1)^d$ where
\begin{align*}
R &= (a_1, b_1) \times(a_2, b_2) \times \cdots \times (a_d, b_d), \\
R' &= (a'_1, b'_1) \times (a'_2, b'_2) \times \cdots \times (a'_d, b'_d), \\
\end{align*}
define $L_{R \to R'} : \mR^d \to \mR^d$ such that
\[
\left[L_{R \to R'}(x) \right]_i := a'_i + \frac{b'_i-a'_i}{b_i-a_i}(x_i-a_i)
\]
for every $i \in [d]$. Notice that $L$ is an affine function that satisfies $\set{ L_{R \to R'}(x) : x \in R} = R'$. Now, given regions $R, R', R'' \subseteq (0,1)^d$, we define
\[
\tn{scale}_{R \to R'}(R'') = \set{L_{R \to R'}(x) : x \in R''}.
\]
We will always apply this scaling function in situations where $R'' \subseteq R$. Thus, intuitively, $\tn{scale}_{R \to R'}(R'')$ returns a region that is contained in $R'$, such that the relative position of $\tn{scale}_{R \to R'}(R'')$ inside $R'$  is the same of that of $R''$ inside $R$. In particular, given a decomposition $D \in \D_{d,p}$, 
\[
\set{ \tn{scale}_{(0,1)^d \to R'}(R) : R \in D}
\]
is a collection of sets that can be obtained from $R'$ by a sequence of $p$-splitting operations.  Conversely, if $\set{R_1,\ldots, R_n}$ is a collection of sets obtained from applying a sequence of $p$-splitting operations to region $R$, then 
\[
\set{ \tn{scale}_{R \to (0,1)^d}(R_i) : i \in [n]}
\]
is an element of $\D_{d,p,n}$. 

We then have the following:

\begin{lemma}\label{MainThm1Lem1}
Let $d \geq 1, p \geq 2$ be fixed integers, and let $y = \sum_{n \geq 1} C_{d,p}(n) x^n$. Then for every subset $S \subseteq [d]$,
\[
\sum_{k \geq 1}  \left| \mathcal{H}_{S, k} \right| x^k = y^{p^{|S|}}.
\]
\end{lemma}

\begin{proof}
For convenience, let $n := p^{|S|}$ throughout this proof. First, observe that
\[
\bigcup_{k \geq p^{|S|}} \mathcal{H}_{S,k} = \set{ D \in \D_{d,p} : D \succeq D_S}.
\]
Next, we show that there is a natural bijection between the set above and $(\D_{d,p})^{n}$: Let $B_1, \ldots, B_n$ be the regions in $D_{S}$. Then, given $n$ decompositions $D_1, \ldots, D_n \in \D_{d,p}$, observe that
\[
D := \bigcup_{j=1}^n \set{ \tn{scale}_{(0,1)^d \to B_j}(R) : R \in D_j}
\]
gives a decomposition of $(0,1)^d$ that refines $D_{S}$. On the other hand, given $D \in \D_{d,p}$ that refines $D_{S}$, define $S_j \subseteq D$ such that
\[
S_j := \set{ R \in D : R \subseteq B_j}
\]
for every $j \in [n]$. By assumption that $D \succeq D_{S}$, we know that $S_1, \ldots, S_n$ partition $D$, as well as that $S_j$ can be obtained from applying a sequence of splitting operations to $B_j$ for every $j \in [n]$. Hence, each of 
\[
D_j := \set{ \tn{scale}_{B_j \to (0,1)^d} (R) : R \in S_j}
\]
is an element of $\D_{d,p}$. This shows that $\left| \mathcal{H}_{S,k} \right|$ is equal to the number of $n$-tuples of decompositions $(D_1, \ldots, D_n) \in (\D_{d,p})^n$ where $\sum_{i=1}^n |D_i| = k$. Thus, it follows that $\left| \mathcal{H}_{S,k} \right| = [x^k] y^{n}$ for every $k \geq 1$, which establishes our claim.
\end{proof}

We are now ready to give our proof for Theorem~\ref{BDTheorem}.

\begin{proof}[Proof of Theorem~\ref{BDTheorem}]
Notice that when $n \geq 2$, every $D \in \D_{d,p,n}$ refines $D_{\set{i}}$ for some coordinate $i$. Thus, we see that $\D_{d,p,n} = \bigcup_{i \in [d]} \mathcal{H}_{ \set{i} ,n}$. Applying the principle of inclusion-exclusion, we obtain that, for every $n \geq 2$:
\[
C_{d,p}(n) = \left| \bigcup_{i \in [d]} \mathcal{H}_{ \set{i}, n} \right|
=  \sum_{S \subseteq [d], |S| \geq 1} (-1)^{|S|+1} \left| \bigcap_{i \in S} \mathcal{H}_{\set{i}, n} \right| 
=  \sum_{S \subseteq [d], |S| \geq 1} (-1)^{|S|+1} \left| \mathcal{H}_{S, n} \right|.
\]
Since the above holds for all $n \geq 2$, we can multiply each side by $x^n$ for each $n$ and sum them up over all $n \geq 2$, and obtain
\begin{equation}\label{MainThm1Eq3}
\sum_{n \geq 2} C_{d,p}(n) x^n = \sum_{n \geq 2} \sum_{S \subseteq [d], |S| \geq 1} (-1)^{|S|+1} \left| \mathcal{H}_{S, n} \right| x^n.
\end{equation}
The left hand side of~\eqref{MainThm1Eq3} is simply equal to $y-x$. For the right hand side, we see that
\begin{align*}
& \sum_{n \geq 2} \sum_{S \subseteq [d], |S| \geq 1} (-1)^{|S|+1} \left| \mathcal{H}_{S, n} \right| x^n \\
&=  \sum_{S \subseteq [d], |S| \geq 1} (-1)^{|S|+1} \left( \sum_{n \geq 2} \left| \mathcal{H}_{S, n} \right| x^n \right) \\
&=  \sum_{S \subseteq [d], |S| \geq 1} (-1)^{|S|+1} y^{p^{|S|}} \\ 
&=  \sum_{k \geq 1}  (-1)^{k+1} \binom{d}{k} y^{p^{k}}.
\end{align*}
Notice that we applied Lemma~\ref{MainThm1Lem1} in the second equality above. Thus,~\eqref{MainThm1Eq3} reduces to
\[
y-x =  \sum_{k \geq 1}  (-1)^{k+1} \binom{d}{k} y^{p^{k}},
\]
which can be easily rearranged to give~\eqref{dgfeq} in the statement of Theorem~\ref{BDTheorem}.
\end{proof}

\subsection{Interchange maximal trees}

Next, we describe a family of trees that is counted by $C_{d,p}(n)$, extending the binary-tree interpretation of the ordinary Catalan numbers. Let $\mathcal{T}_{d,p,n}$ denote the set of full $p$-ary trees 
(i.e., every internal node has exactly $p$ children) with $n$ (unlabelled) leaves, 
such that each of the $m =\frac{n-1}{p-1}$ internal nodes is assigned a label from $[d]$. Also, for convenience, we define
\[
B_{i,j} := \set{ x \in (0,1)^d : \frac{j-1}{p} < x_i < \frac{j}{p}}
\]
for every $i,j \in [p]$. Note that $\set{B_{i,1}, \ldots, B_{i,p}}$ are exactly the sets obtained from $p$-splitting $(0,1)^d$ in coordinate $i$. 

The following describes a mapping from a tree in $\T_{d,p,n}$ to a hypercube decomposition in $\D_{d,p,n}$.

\begin{definition}\label{treetodecomp}
Define the function $f \colon \T_{d,p,n} \to \D_{d,p,n}$ recursively as follows:
\begin{itemize}
\item
($n=1$)
$f$ maps the exceptional tree with a single node to $\set{(0,1)^d}$,
the trivial decomposition with a single region.
\item
($n \ge 2$) Given $T \in \T_{d,p,n}$ where its root node is labelled $i$ hand has subtrees (from left to right) 
$T_{1}, \ldots, T_p$, define
\[
f(T) = \bigcup_{j=1}^p \set{ \tn{scale}_{(0,1)^d \to B_{i,j}}(R) : R \in f(T_j) }.
\]
That is, we take the decompositions $f(T_1), \ldots, f(T_p)$, and scale each of them to fit into the sets $B_{i,1}, \ldots, B_{i,p}$ respectively. Then $f(T)$ --- the union of these $p$ scaled decompositions --- would be a decomposition of $(0,1)^d$ in its own right.
\end{itemize}
Also, we define two trees $T,T'$ to be \emph{interchange equivalent} if $f(T) = f(T')$. 
\end{definition}

Figure~\ref{figtreetodecomp} illustrates Definition~\ref{treetodecomp}:
It displays three trees in $\T_{3,2,6}$ that are mapped by $f$ 
to the same decomposition in $\D_{3,2,6}$, and hence shows that $f$ is not one-to-one.

\begin{figure}[htb]
\begin{center}
\begin{tabular}{c@{\qquad}c@{\qquad}c}
$T_1$ & $T_2$ & $T_3$
\\[1mm]
\begin{tikzpicture}
[scale=0.35,thick,main node/.style={circle,inner sep=0.5mm,draw,font=\scriptsize\sffamily}]
  \node[main node] at (4,6) (0) {$1$};
  \node[main node] at (2,4) (l){$2$};
  \node[main node] at (6,4) (r){$2$};
  \node[main node] at (1,2) (ll){};
  \node[main node] at (3,2) (lr){};
  \node[main node] at (5,2) (rl){$3$};
  \node[main node] at (7,2) (rr){$3$};
  \node[main node] at (4.5,0) (rll){};
  \node[main node] at (5.5,0) (rlr){};
  \node[main node] at (6.5,0) (rrl){};
  \node[main node] at (7.5,0) (rrr){};
  \path[every node/.style={font=\sffamily}]
    (0) edge (l)
    (0) edge (r)
    (l) edge (ll)
    (l) edge (lr)
    (r) edge (rl)
    (r) edge (rr)
    (rl) edge (rll)
    (rl) edge (rlr)
    (rr) edge (rrl)
    (rr) edge (rrr)
    ;
\end{tikzpicture}
&
\begin{tikzpicture}
[scale=0.35,thick,main node/.style={circle,inner sep=0.5mm,draw,font=\scriptsize\sffamily}]
  \node[main node] at (4,6) (0) {$1$};
  \node[main node] at (2,4) (l){$2$};
  \node[main node] at (6,4) (r){$3$};
  \node[main node] at (1,2) (ll){};
  \node[main node] at (3,2) (lr){};
  \node[main node] at (5,2) (rl){$2$};
  \node[main node] at (7,2) (rr){$2$};
  \node[main node] at (4.5,0) (rll){};
  \node[main node] at (5.5,0) (rlr){};
  \node[main node] at (6.5,0) (rrl){};
  \node[main node] at (7.5,0) (rrr){};
  \path[every node/.style={font=\sffamily}]
    (0) edge (l)
    (0) edge (r)
    (l) edge (ll)
    (l) edge (lr)
    (r) edge (rl)
    (r) edge (rr)
    (rl) edge (rll)
    (rl) edge (rlr)
    (rr) edge (rrl)
    (rr) edge (rrr)
    ;
\end{tikzpicture}
&
\begin{tikzpicture}
[scale=0.35,thick,main node/.style={circle,inner sep=0.5mm,draw,font=\scriptsize\sffamily}]
  \node[main node] at (4,6) (0) {$2$};
  \node[main node] at (2,4) (l){$1$};
  \node[main node] at (6,4) (r){$1$};
  \node[main node] at (1,2) (ll){};
  \node[main node] at (3,2) (lr){$3$};
  \node[main node] at (5,2) (rl){};
  \node[main node] at (7,2) (rr){$3$};
  \node[main node] at (2.5,0) (lrl){};
  \node[main node] at (3.5,0) (lrr){};
  \node[main node] at (6.5,0) (rrl){};
  \node[main node] at (7.5,0) (rrr){};
  \path[every node/.style={font=\sffamily}]
    (0) edge (l)
    (0) edge (r)
    (l) edge (ll)
    (l) edge (lr)
    (r) edge (rl)
    (r) edge (rr)
    (lr) edge (lrl)
    (lr) edge (lrr)
    (rr) edge (rrl)
    (rr) edge (rrr)
    ;
\end{tikzpicture}
\\[6mm]
\multicolumn{3}{c}{
$f(T_1)=f(T_2)=f(T_3)=$
\adjustbox{valign=m}{
\begin{tikzpicture}
[scale=0.3,main node/.style={circle,draw,font=\scriptsize\sffamily},font = \scriptsize]
\def\ys{2} 
\draw[thick, ->] (0,0) -- (12.5,0) node[anchor = west] {$1$};
\draw[thick, ->] (0,0) -- (3, {3*\ys}) node[anchor = south] {$2$};
\draw[thick, ->] (0,0) -- (0,12.5) node[anchor = south] {$3$};
\draw (0,0) -- (10,0) -- ( 12, {2* \ys}) -- (2,{2* \ys}) -- (0,0);
\draw (0,10) -- (10,10) -- (12,{10 + 2* \ys}) -- (2,{10 + 2* \ys}) -- (0,10);
\draw (0,0) -- (0,10);
\draw (10,0) -- (10,10);
\draw (12,{2* \ys}) -- (12,{10 + 2* \ys});
\draw (2,{2* \ys}) -- (2,{10 + 2* \ys});
\draw[dashed] (5,0) -- (7,{2* \ys}) -- (7,{10 + 2* \ys}) -- (5,10) -- (5,0);
\draw[dashed] (1,{\ys}) -- (11,{\ys}) -- (11,{10+\ys}) -- (1,{10+\ys}) -- (1,{\ys});
\draw[dashed] (6,{\ys}) -- (6,{10+\ys});
\draw[dashed] (5,5) -- (10,5) -- (12, {5+2*\ys}) -- (7,{5+2*\ys}) -- (5,5);
\draw[dashed] (6,{5+\ys}) -- (11,{5+\ys});
\end{tikzpicture}
}}
\end{tabular}
\caption{Example of the map $f$ from trees to decompositions (Definition~\ref{treetodecomp})}\label{figtreetodecomp}
\end{center}
\end{figure}

Figure~\ref{figinterchangeequiv} describes how to produce interchange equivalent trees. 
Suppose $T$ is a tree whose root has label $i$, and all $p$ children of the root node are internal nodes with the same label $j \neq i$. 
Let $T_{11}, T_{12}, \ldots, T_{pp}$ denote the $p^2$ subtrees from left to right 
of the $p$ nodes labelled $j$ (top of Figure~\ref{figinterchangeequiv}). Then, by the definition of $f$, we have
\begin{align*}
f(T) &= \bigcup_{k \in [p]} \bigcup_{\ell \in [p]} \tn{scale}_{(0,1)^d \to B_{i,k}}\left(\tn{scale}_{(0,1)^d \to B_{j, \ell}}(f(T_{k\ell})) \right)\\
&= \bigcup_{k, \ell \in [p]} \tn{scale}_{(0,1)^d \to B_{i,k} \cap B_{j,\ell}} (f(T_{k\ell})).
\end{align*}
Next, if we let $T'$ be the tree at the bottom of Figure~\ref{figinterchangeequiv}, then by the same rationale one could show that $f(T')$ yields the same collection of sets as above for $f(T)$. Thus, $T,T'$ are interchange equivalent. More generally, if $T$ was a subtree of a larger tree $\overline{T}$ and we let $\overline{T}'$ be the tree obtained from $\overline{T}$ by replacing the subtree $T$ by $T'$, then it is easy to see that $f(\overline{T}) = f(\overline{T}')$. Notice that the three trees in Figure~\ref{figtreetodecomp} can be transformed into each other using this ``subtree swapping'' process that preserves interchange equivalence.

\begin{figure}[htb]
\begin{center}
\begin{tabular}{c}
\begin{tikzpicture}
[scale =0.67, thick, main node/.style={circle, inner sep =0.5mm,draw,  font=\footnotesize\sffamily}, tree node/.style={draw, inner sep =0.7mm, font=\footnotesize\sffamily}]

  \node[main node] at ( 6,6) (0) {$i$};
  \node[main node] at ( 0,4) (1) {$j$};
  \node[main node] at ( 4,4) (2) {$j$};
  \node[main node] at (12,4) (p) {$j$};
  \node[tree node] at (-1,2) (11){$T_{11}$};
  \node[tree node] at ( 1,2) (1p){$T_{1p}$};
  \node[tree node] at ( 3,2) (21){$T_{21}$};
  \node[tree node] at ( 5,2) (2p){$T_{2p}$};
  \node[tree node] at (11,2) (p1){$T_{p1}$};
  \node[tree node] at (13,2) (pp){$T_{pp}$};

  \path[every node/.style={font=\sffamily}]
    (0) edge (1)
    (0) edge (2)
    (0) edge (p)
    (1) edge (11)
    (1) edge (1p)
    (2) edge (21)
    (2) edge (2p)
    (p) edge (p1)
    (p) edge (pp);
    
    \node at ( 8,4) {$\boldsymbol{\cdots\cdots}$};
    \node at ( 8,2) {$\boldsymbol{\cdots\cdots}$};
    \node at ( 4,2) {$\boldsymbol{\cdots}$};
    \node at ( 0,2) {$\boldsymbol{\cdots}$};
    \node at (12,2){$\boldsymbol{\cdots}$};
\end{tikzpicture}

\\[5mm]

\begin{tikzpicture}
[scale =0.67, thick, main node/.style={circle, inner sep =0.5mm,draw,  font=\footnotesize\sffamily}, tree node/.style={draw, inner sep =0.7mm, font=\footnotesize\sffamily}]

  \node[main node] at ( 6,6) (0) {$j$};
  \node[main node] at ( 0,4) (1) {$i$};
  \node[main node] at ( 4,4) (2) {$i$};
  \node[main node] at (12,4) (p) {$i$};
  \node[tree node] at (-1,2) (11){$T_{11}$};
  \node[tree node] at ( 1,2) (1p){$T_{p1}$};
  \node[tree node] at ( 3,2) (21){$T_{12}$};
  \node[tree node] at ( 5,2) (2p){$T_{p2}$};
  \node[tree node] at (11,2) (p1){$T_{1p}$};
  \node[tree node] at (13,2) (pp){$T_{pp}$};

  \path[every node/.style={font=\sffamily}]
    (0) edge (1)
    (0) edge (2)
    (0) edge (p)
    (1) edge (11)
    (1) edge (1p)
    (2) edge (21)
    (2) edge (2p)
    (p) edge (p1)
    (p) edge (pp);
    
    \node at ( 8,4) {$\boldsymbol{\cdots\cdots}$};
    \node at ( 8,2) {$\boldsymbol{\cdots\cdots}$};
    \node at ( 4,2) {$\boldsymbol{\cdots}$};
    \node at ( 0,2) {$\boldsymbol{\cdots}$};
    \node at (12,2){$\boldsymbol{\cdots}$};
\end{tikzpicture}

\end{tabular}
\vspace{1mm}
\caption{Two interchange equivalent trees}\label{figinterchangeequiv}
\end{center}
\end{figure}

Next, we consider a map that acts as an inverse of $f$ by choosing a unique representative in each inverse image $f^{-1}(D)$ for $D \in \D_{d,p,n}$. Intuitively, given a decomposition, we construct the corresponding tree by starting from the root and iteratively choosing the maximum possible node label. More precisely:

\begin{definition}\label{decomptotree}
Define the function $g: \D_{d,p,n} \to \T_{d,p,n}$ recursively as follows:
\begin{itemize}
\item
($n=1$) $g$ maps the trivial decomposition $\set{ (0,1)^d}$ to the tree with a single node.
\item
($n \geq 2$) Given a decomposition $D\in \D_{d,p,n}$, we choose the largest index $i$ such that every region in $D$ is contained in $B_{i,j}$ for some $j$. Next, define $R_j = \set{ R \in D, R \subseteq B_{i,j}}$ for every $j \in [p]$. By the choice of $i$, $R_1, \ldots, R_p$ must partition $D$, and 
\[
D_j := \set{ \tn{scale}_{B_{i,j} \to (0,1)^d} (R) : R \in R_j}
\]
is a decomposition of the unit hypercube in its own right, for every $j \in [p]$.

We then define $g(D)$ to be the $p$-ary tree with root node labelled $i$, with subtrees $g(D_1), \ldots, g(D_p)$ from left to right.
\end{itemize}
\end{definition}

\begin{example}
Let $D$ be the decomposition in Figure~\ref{figtreetodecomp}, which has 6 regions:
\begin{alignat*}{2}
D = 
\{ \;
&(0,1/2) \times (0, 1/2) \times (0,1),   &\quad &(0,1/2) \times (1/2,1) \times (0,1), 
\\
&(1/2,1) \times (0, 1/2) \times (0,1/2), &\quad &(1/2,1) \times (1/2, 1) \times (0,1/2), 
\\
&(1/2,1) \times (0, 1/2) \times (1/2,1), &\quad &(1/2,1) \times (1/2, 1) \times (1/2,1)  
\;\; \}. 
\end{alignat*}
Every region in $D$ is contained in one of the sets
\[
B_{1,1} =  (0,1/2) \times (0,1) \times (0,1), \quad B_{1,2} = (1/2,1) \times (0,1) \times (0,1).
\]
However, the same can be said for 
\[
B_{2,1} =  (0,1) \times (0,1/2) \times (0,1), \quad B_{2,2} = (0,1) \times (1/2,1) \times (0,1).
\]
but \emph{not} for the sets $B_{3,1}, B_{3,2}$. 
Thus, we choose the index $i=2$, and so the root node of $g(D)$ has label $2$. 
Continuing in this way, one finds that $g(D)$ is the tree $T_3$ in Figure~\ref{figtreetodecomp}.
\end{example}

By the construction of the functions $f$ and $g$, it is clear that $g$ is a one-to-one function, and that $f(g(D)) = D$ for every $D \in \D_{d,p,n}$. Hence $f$ is onto, which implies that $C_{d,p}(n)$ counts 
the number of interchange equivalence classes in $\T_{d,p,n}$. In other words, $C_{d,p}(n)$ counts the number of trees in the image of $g$. We characterize these trees through the following definition:

\begin{definition}\label{defninterchangemax}
A tree $T \in \T_{d,p,n}$ is \emph{interchange maximal} if it has no subtree $\tilde{T}$ such that
\begin{enumerate}
\item[(T1)]
$\tilde{T}$ has root labelled $i$;
\item[(T2)]
there exists $\tilde{T'}$ with root labelled $j$ where $j > i$ such that $f(\tilde{T}') = f(\tilde{T})$.
\end{enumerate}
\end{definition}

\begin{example}
In Figure~\ref{figtreetodecomp}, $T_1$ and $T_2$ are not interchange maximal 
since they both have root labelled $i=1$, while there exists an interchange equivalent tree $T_3$ 
with root node labelled $j=2$. 
One can check that $T_3$ is indeed interchange maximal. 
\end{example}

Since the function $g$ always picks the largest possible node label when generating the tree from the top down, it follows immediately from Definitions~\ref{decomptotree} and~\ref{defninterchangemax} that if $T \in \T_{d,p,n}$ is in the image of $g$, then $T$ must be interchange maximal. We show that the converse is also true.

\begin{lemma}\label{uniqueintmax}
If $T \in \T_{d,p,n}$ is interchange maximal, then there exists $D \in \D_{d,p,n}$ where $g(D) = T$.
\end{lemma}

\begin{proof}
Suppose for a contradiction that there exists an interchange maximal tree $T \in \T_{d,p,n}$ that is not in the image of $g$. Moreover, choose such a tree $T$ where $n$ is minimized. Then, there exists a distinct tree $T' := g(f(T))$ where $f(T) = f(T')$. Since $T'$ is in the image of $g$, it must be interchange maximal as well.

Next, suppose the root node of $T$ has label $i$ (the root node of $T$ must be an internal node as the single-node tree is indeed in the image of $g$).  Then it follows that the root node of $T'$ must also have label $i$, otherwise the tree with the smaller root node label would not be interchange maximal.

Next, let $T_1, \ldots, T_p$ be the subtrees of the root of $T$ from left to right, and likewise let $T'_1, \ldots, T'_p$ be subtrees of the root of $T'$ from left to right. Since $f(T) = f(T')$ and their root nodes both have label $i$, it follows that $f(T_j) = f(T'_j)$ for all $j \in [p]$. Also, since $T \neq T'$, there must exist some index $\ell$ where $T_{\ell} \neq T'_{\ell}$.

However, It is clear from Definition~\ref{defninterchangemax} that any subtree of an interchange maximal tree is also interchange maximal. Thus, we have just found two interchange maximal trees $T_{\ell}, T'_{\ell}$ that are interchange equivalent, which implies that at least one of $T_{\ell}, T'_{\ell}$ is not in the image of $g$. This contradicts the minimality of $n$.
\end{proof}

Thus, through Lemma~\ref{uniqueintmax} and the immediately preceding discussion, we have shown the following:

\begin{theorem}
For every integer $d \geq 1, p \geq 2$, and $n \geq 1$, $C_{d,p}(n)$ counts the number of interchange maximal trees in $\T_{d,p,n}$.
\end{theorem}

Finally, notice that condition (T2) in Definition~\ref{defninterchangemax} involves the tree-to-decomposition function $f$. We conclude this section by showing an alternative characterization of interchange maximal trees that is graph theoretical and does not make references to decompositions.

\begin{proposition}
Condition (T2) in Definition~\ref{defninterchangemax} is equivalent to
\begin{enumerate}
\item[(T2')]
There exists index $j > i$ such that every path joining the root of $\tilde{T}$ and a leaf node of $\tilde{T}$ contains an internal node labelled $j$.
\end{enumerate}
\end{proposition}

\begin{proof}
We first prove that (T2') implies (T2). Suppose we are given a tree $T$ with a subtree $\tilde{T}$ that satisfies (T2'). Let $u$ denote the root node of $\tilde{T}$. Given a leaf node $v$ in $\tilde{T}$, consider the unique path in $\tilde{T}$ joining $u$ and $v$, and let $d(v)$ denote the distance between $u$ and the closest $j$-node on this path. Given (T2'), $d(v)$ is well-defined and finite for every leaf node $v$. 

Let $d_{\max} := \max \set{d(v) : \tn{$v$ is a leaf of $\tilde{T}$}}$. We prove that (T2) holds by induction on $d_{\max}$. If $d_{\max} = 1$, then every child of $u$ is an internal node labelled $j$, and one can perform the subtree swapping operation described in Figure~\ref{figinterchangeequiv} to find a tree $\tilde{T}'$ with root node label $j$ where $f(\tilde{T}) = f(\tilde{T}')$, establishing (T2). For the inductive step, suppose $d_{\max} = \ell > 1$. Let $v$ be a leaf node where $d(v) = \ell$, and let $u_0, \ldots, u_k$ be unique path in $\tilde{T}$ going from $u$ to $v$. Thus, we know that $u=u_0, v=u_k$, $u_{\ell}$ is an internal node labelled $j$, and none of $u_1, \ldots, u_{\ell - 1}$ has label $j$. Now consider the node $u_{\ell-1}$ and its $p$ children. If $w$ is a child of $u_{\ell-1}$ and $w$ was not an internal node labelled $j$, then let $v'$ be a leaf in $\tilde{T}$ that is a descendent of $w$ (possibly $w$ itself). Now observe that the $uv'$ path in $\tilde{T}$ contains all of the vertices $ u_1,  \ldots, u_{\ell-1}$, and $w$, none of which is an internal node labelled $j$. This means that $d(v') > d(v) =  \ell$, contradicting our choice of $v$. Thus, all $p$ children of $u_{\ell-1}$ are internal nodes labelled $j$, and again we can apply the tree swapping operation to obtain an interchange equivalent tree $\tilde{T}'$ where the node $u_{\ell-1}$ now has label $j$. If we iteratively perform this operation on all leaves $v$ with $d(v) = \ell$, we will eventually obtain an interchange equivalent tree with $d_{\max}$ lower than $\ell$, completing the induction.

Next, we prove that (T2) implies (T2'). For this part of the proof, it is convenient to extend the tree-to-decomposition mapping $f$ as follows. Notice that for each leaf node $v$ of a given tree $T$ naturally corresponds to a region in the decomposition $f(T)$. Thus, given a leaf node $v$ in a tree $T$, let $f(T, v)$ denote the region in $f(T)$ that corresponds to node $v$. Now, suppose we have a subtree $\tilde{T}$ where (T2) holds, and so there is a tree $\tilde{T}'$ with root node label $j$ and $f(\tilde{T}') = f(\tilde{T})$. Thus,
\[
\set{ f(\tilde{T}, v) : \tn{$v$ a leaf of $\tilde{T}$} } = \set{ f(\tilde{T}', v') : \tn{$v'$ a leaf of $\tilde{T}'$} }.
\]
In particular, since $\tilde{T}'$ has root node label $j$, we are assured that for every leaf $v$ of $\tilde{T}$, $f(\tilde{T},v)$ is a subset of one of $B_{j,1}, \ldots, B_{j,p}$. Thus, the sequence of splitting operations that decomposes $(0,1)^d$ into region $f(\tilde{T},v)$ must involve splitting in coordinate $j$ at some point. This means that, for every leaf $v$ of $\tilde{T}$, the path joining the root of $\tilde{T}$ and $v$ must contain an internal node with label $j$, proving (T2').
\end{proof}

For an example, condition (T2') provides an easy way to see that the tree $T_2$ in Figure~\ref{figtreetodecomp} is not interchange maximal, since it has root node label $i=1$, while every path joining this node and a leaf node of the tree contains a node labelled $j=2$.

\section{Future research directions}
\label{sec5}

\subsection{Wedderburn--Etherington numbers}

This paper has developed a more complete understanding of the numbers $C_{d,p}(n)$ 
and established a connection between 
hypercube decompositions and interchange maximal trees.
It is an important open problem to determine if a similar theory can be developed 
for other well-known variations on the Catalan numbers.
In particular, we consider the Wedderburn--Etherington numbers \cite{Etherington37,Wedderburn22};
see also OEIS \seqnum{A001190}.
In this case we have a binary operation that is commutative but not associative;
we let $W(n)$ denote the number of ways to interpret $x^n$ under this operation. 
For instance, for $x^4$ we have
\[
((xx)x)x = (x(xx))x = x(x(xx)) = x((xx)x),
\]
but all of these are distinct from $(xx)(xx)$. 
Thus there are only two distinct interpretations of $x^4$, and so $W(4) = 2$. 
The following table lists the distinct $n^{\tn{th}}$ powers for $n \leq 5$. 
\begin{center}
\begin{tabular}{c|c|c}
$n$ & $n^{\tn{th}}$ commutative nonassociative powers & $W(n)$\\
\hline
$1$ & $x$ & $1$\\
$2$ & $xx$ & $1$\\
$3$ & $(xx)x$ & $1$\\
$4$ & $((xx)x)x, \;\; (xx)(xx)$ & $2$ \\
$5$ & $((xx)x)x)x, \;\; ((xx)xx))x, \;\; ((xx)x)(xx)$ & $3$\\
\end{tabular}
\end{center}
The growth rate of $W(n)$ has been determined~\cite{Landau77}, 
and more recently a non-recursive formula for these numbers has been found \cite[Theorem 2]{Billey}.
It would be interesting to investigate higher-dimensional analogues 
of the Wedderburn--Etherington numbers involving $d$ distinct $p$-ary operations
satisfying various generalizations of commutativity.

\subsection{Algebraic operads}

The last decade or two has seen the development of a new theory of
operadic combinatorics \cite{Bacher,Chapoton,Giraudo,Loday},
which constructs algebraic operads in terms of combinatorial objects,
applies the theory of operads to illuminate combinatorial structures,
and uncovers new integer sequences which deserve further combinatorial study.
The original motivation for the present paper came from combinatorial aspects 
of work by the third author and Dotsenko \cite{BD} on Boardman--Vogt tensor products 
of absolutely free operads.

In this section revisit the trees in Figure~\ref{figtreetodecomp}, and give a brief explanation for their equivalence in the operadic perspective. This discussion assumes familiarity with some of the basic terminology from the theory of Gr\"obner bases for operads \cite{DK}; see also \cite{BDbook}.

In the case of $d = 3$ and $p=2$, we have three binary operations $\bigcdot_1$, $\bigcdot_2$, $\bigcdot_3$
satisfying the three interchange laws:
\begin{align*}
&( a\, \bigcdot_2\, b ) \,\bigcdot_1\, ( c\, \bigcdot_2 \,d ) = ( a\, \bigcdot_1 \,c ) \,\bigcdot_2\, ( b\, \bigcdot_1 \,d ),
\\
&( a\, \bigcdot_3\, b ) \,\bigcdot_1\, ( c\, \bigcdot_3\, d ) = ( a\, \bigcdot_1\, c ) \,\bigcdot_3\, ( b\, \bigcdot_1 \,d ),
\\
&( a \,\bigcdot_3\, b ) \,\bigcdot_2\, ( c\, \bigcdot_3\, d ) = ( a \,\bigcdot_2\, c ) \,\bigcdot_3\, ( b\, \bigcdot_2\, d ).
\end{align*}

The operation order $\bigcdot_1 \prec  \bigcdot_2 \prec \bigcdot_3$ extends to 
a monomial order on operad ideal generated by these three relations.
Then the three interchange laws above can be represented by the following tree polynomials:
\begin{align*}
\alpha
\;\;
&=
\adjustbox{valign=m}{
\begin{tikzpicture}[scale =0.35, thick,main node/.style={circle, minimum size=0.4cm, inner sep =0.1mm, draw,  font=\scriptsize\sffamily}]
  \node[main node] at (4,6) (0) {$1$};
  \node[main node] at (2,4) (l){$2$};
    \node[main node] at (6,4) (r){$2$};
  \node[main node] at (1,2) (ll){$a$};
    \node[main node] at (3,2) (lr){$b$};
  \node[main node] at (5,2) (rl){$c$};
    \node[main node] at (7,2) (rr){$d$};
  \path[every node/.style={font=\sffamily}]
    (0) edge (l)
    (0) edge (r)
    (l) edge (ll)
 (l) edge (lr)
    (r) edge (rl)
    (r) edge (rr)  
    ;
\end{tikzpicture}
}
-
\adjustbox{valign=m}{
\begin{tikzpicture}[scale =0.35, thick,main node/.style={circle,  minimum size=0.4cm, inner sep =0.1mm, draw,  font=\scriptsize\sffamily}]
  \node[main node] at (4,6) (0) {$2$};
  \node[main node] at (2,4) (l){$1$};
    \node[main node] at (6,4) (r){$1$};
  \node[main node] at (1,2) (ll){$a$};
    \node[main node] at (3,2) (lr){$c$};
  \node[main node] at (5,2) (rl){$b$};
    \node[main node] at (7,2) (rr){$d$};
  \path[every node/.style={font=\sffamily}]
    (0) edge (l)
    (0) edge (r)
    (l) edge (ll)
 (l) edge (lr)
    (r) edge (rl)
    (r) edge (rr)  
    ;
\end{tikzpicture}
}
\\
\beta
\;\;
&=
\adjustbox{valign=m}{
\begin{tikzpicture}[scale =0.35, thick,main node/.style={circle,  minimum size=0.4cm, inner sep =0.1mm, draw,  font=\scriptsize\sffamily}]
  \node[main node] at (4,6) (0) {$1$};
  \node[main node] at (2,4) (l){$3$};
    \node[main node] at (6,4) (r){$3$};
  \node[main node] at (1,2) (ll){$a$};
    \node[main node] at (3,2) (lr){$b$};
  \node[main node] at (5,2) (rl){$c$};
    \node[main node] at (7,2) (rr){$d$};
  \path[every node/.style={font=\sffamily}]
    (0) edge (l)
    (0) edge (r)
    (l) edge (ll)
 (l) edge (lr)
    (r) edge (rl)
    (r) edge (rr)  
    ;
\end{tikzpicture}
}
-
\adjustbox{valign=m}{
\begin{tikzpicture}[scale =0.35, thick,main node/.style={circle,  minimum size=0.4cm, inner sep =0.1mm, draw,  font=\scriptsize\sffamily}]
  \node[main node] at (4,6) (0) {$3$};
  \node[main node] at (2,4) (l){$1$};
    \node[main node] at (6,4) (r){$1$};
  \node[main node] at (1,2) (ll){$a$};
    \node[main node] at (3,2) (lr){$c$};
  \node[main node] at (5,2) (rl){$b$};
    \node[main node] at (7,2) (rr){$d$};
  \path[every node/.style={font=\sffamily}]
    (0) edge (l)
    (0) edge (r)
    (l) edge (ll)
 (l) edge (lr)
    (r) edge (rl)
    (r) edge (rr)
    ;
\end{tikzpicture}
}
\\
\gamma
\;\;
&=
\adjustbox{valign=m}{
\begin{tikzpicture}[scale =0.35, thick,main node/.style={circle,  minimum size=0.4cm, inner sep =0.1mm, draw,  font=\scriptsize\sffamily}]
  \node[main node] at (4,6) (0) {$2$};
  \node[main node] at (2,4) (l){$3$};
    \node[main node] at (6,4) (r){$3$};
  \node[main node] at (1,2) (ll){$a$};
    \node[main node] at (3,2) (lr){$b$};
  \node[main node] at (5,2) (rl){$c$};
    \node[main node] at (7,2) (rr){$d$};
  \path[every node/.style={font=\sffamily}]
    (0) edge (l)
    (0) edge (r)
    (l) edge (ll)
 (l) edge (lr)
    (r) edge (rl)
    (r) edge (rr)
    ;
\end{tikzpicture}
}
-
\adjustbox{valign=m}{
\begin{tikzpicture}[scale =0.35, thick,main node/.style={circle, minimum size=0.4cm, inner sep =0.1mm, draw,  font=\scriptsize\sffamily}]
  \node[main node] at (4,6) (0) {$3$};
  \node[main node] at (2,4) (l){$2$};
    \node[main node] at (6,4) (r){$2$};
  \node[main node] at (1,2) (ll){$a$};
    \node[main node] at (3,2) (lr){$c$};
  \node[main node] at (5,2) (rl){$b$};
    \node[main node] at (7,2) (rr){$d$};
  \path[every node/.style={font=\sffamily}]
    (0) edge (l)
    (0) edge (r)
    (l) edge (ll)
 (l) edge (lr)
    (r) edge (rl)
    (r) edge (rr)
    ;
\end{tikzpicture}
}
\end{align*}
The least common multiple of the leading tree monomials of $\alpha$ and $\gamma$ is
\[
[\alpha,\gamma]
=
\adjustbox{valign=m}{
\begin{tikzpicture}[scale =0.35, thick,main node/.style={circle, minimum size=0.4cm, inner sep =0.1mm, draw, font=\scriptsize\sffamily}]
  \node[main node] at (4,6) (0) {$1$};
  \node[main node] at (2,4) (l){$2$};
    \node[main node] at (6,4) (r){$2$};
  \node[main node] at (1,2) (ll){$a$};
    \node[main node] at (3,2) (lr){$b$};
  \node[main node] at (4.5,2) (rl){$3$};
    \node[main node] at (7.5,2) (rr){$3$};
  \node[main node] at (3,0) (rll){$c$};
    \node[main node] at (5,0) (rlr){$d$};
  \node[main node] at (6.5,0) (rrl){$e$};
    \node[main node] at (8.5,0) (rrr){$f$};
  \path[every node/.style={font=\sffamily}]
    (0) edge (l)
    (0) edge (r)
    (l) edge (ll)
 (l) edge (lr)
    (r) edge (rl)
    (r) edge (rr)
        (rl) edge (rll)
    (rl) edge (rlr)
        (rr) edge (rrl)
    (rr) edge (rrr)
    ;
    ;
\end{tikzpicture}
}
\]
It follows that the operad ideal generated by $\alpha$ and $\gamma$ contains
the following tree polynomial:
\[
\epsilon
\;
=
\;
\adjustbox{valign=m}{
\begin{tikzpicture}[scale =0.35, thick,main node/.style={circle,  minimum size=0.4cm, inner sep =0.1mm, draw,  font=\scriptsize\sffamily}]
  \node[main node] at (4,6) (0) {$1$};
  \node[main node] at (2,4) (l){$2$};
    \node[main node] at (6,4) (r){$3$};
  \node[main node] at (1,2) (ll){$a$};
    \node[main node] at (3,2) (lr){$b$};
  \node[main node] at (4.5,2) (rl){$2$};
    \node[main node] at (7.5,2) (rr){$2$};
  \node[main node] at (3,0) (rll){$c$};
    \node[main node] at (5,0) (rlr){$e$};
  \node[main node] at (6.5,0) (rrl){$d$};
    \node[main node] at (8.5,0) (rrr){$f$};
  \path[every node/.style={font=\sffamily}]
    (0) edge (l)
    (0) edge (r)
    (l) edge (ll)
 (l) edge (lr)
    (r) edge (rl)
    (r) edge (rr)
        (rl) edge (rll)
    (rl) edge (rlr)
        (rr) edge (rrl)
    (rr) edge (rrr)
    ; 
    ;
\end{tikzpicture}
}
\!\!\!
-
\;
\adjustbox{valign=m}{
\begin{tikzpicture}[scale =0.35, thick,main node/.style={circle,  minimum size=0.4cm, inner sep =0.1mm,draw,  font=\scriptsize\sffamily}]
  \node[main node] at (4,6) (0) {$2$};
  \node[main node] at (2,4) (l){$1$};
    \node[main node] at (6,4) (r){$1$};
  \node[main node] at (1,2) (ll){$a$};
    \node[main node] at (3,2) (lr){$3$};
  \node[main node] at (5,2) (rl){$b$};
    \node[main node] at (7,2) (rr){$3$};
  \node[main node] at (2,0) (lrl){$c$};
    \node[main node] at (4,0) (lrr){$d$};
  \node[main node] at (6,0) (rrl){$e$};
    \node[main node] at (8,0) (rrr){$f$};
  \path[every node/.style={font=\sffamily}]
    (0) edge (l)
    (0) edge (r)
    (l) edge (ll)
 (l) edge (lr)
    (r) edge (rl)
    (r) edge (rr)
        (lr) edge (lrl)
    (lr) edge (lrr)
        (rr) edge (rrl)
    (rr) edge (rrr)
    ; 
    ;
\end{tikzpicture}
}
\]
The tree polynomial $\epsilon$ is the simplest new element in 
the operadic Gr\"obner basis beyond the ideal generators $\alpha$, $\beta$, and $\gamma$.
The two terms in $\epsilon$ correspond to the two trees $T_2$ and $T_3$ in Figure~\ref{figtreetodecomp}.
This provides an operadic explanation of the fact that the distinct trees
$T_2$ and $T_3$ produce the same decomposition of the cube.
It remains an open (and probably very difficult) problem to compute 
a complete Gr\"obner basis for this operad ideal.

\section{Acknowledgments}
\label{sec7}

The research of the second and third authors was supported by a Discovery Grant
from NSERC, the Natural Sciences and Engineering Research Council of Canada. 
The authors thank Chris Soteros for very helpful comments on analytic combinatorics, 
which helped us assemble the necessary tools to prove Theorem~\ref{Cdpnasymp}. We are also extremely thankful for the anonymous referees who provided feedback that helped improve the presentation of the paper. Finally, we express our gratitude to Neil Sloane and the numerous contributors 
to the On-Line Encyclopedia of Integer Sequences (OEIS) for developing and curating 
such a valuable resource for mathematical research.

\bigskip
\hrule
\bigskip

\noindent 
2010 \emph{Mathematics Subject Classification}:~Primary 05A15. 
Secondary 05A16, 05A19, 17A42, 17A50, 18D50, 52C17.

\medskip

\noindent 
\emph{Keywords}:~Catalan number,
higher-dimensional analog, 
hypercube decomposition, 
enumeration problem,
generating function,
Lagrange inversion,
interchange law, 
algebraic operad,
Boardman--Vogt tensor product,
nonassociative algebra,
higher category. 

\bigskip
\hrule
\bigskip

\noindent 
Concerned with sequences
\seqnum{A000108},
\seqnum{A001190},
\seqnum{A236339},
\seqnum{A236342}.

\bigskip
\hrule
\bigskip

\end{document}